\documentclass[11pt,a4paper]{article}
\usepackage[english]{babel}       \usepackage[latin1]{inputenc}
\usepackage[T1]{fontenc}          \usepackage{a4}
\usepackage[cyr]{aeguill}         \usepackage{epsfig}
\usepackage{amsmath,amsthm,amsfonts,amssymb,mathrsfs}
\usepackage{dsfont}                \usepackage{float}
\usepackage{fancyhdr}              \usepackage{varioref}
\usepackage{color}
\usepackage{url}                   \usepackage{calc}
\usepackage{wrapfig}               \usepackage{graphicx}
\usepackage{enumitem,hyperref}
\usepackage{pifont}       \usepackage{tikz}
\newtheorem{theorem}{Theorem} 
\newtheorem{definition}{Definition}[section]
\newtheorem{remark}{Remark}[section]
\newtheorem{proposition}{Proposition}[section]
\newtheorem{lemma}{Lemma}[section]

\newtheorem{corollary}{Corollary}[section]

\numberwithin{equation}{section}
\usepackage[left,modulo]{lineno}   

\usepackage{authblk}

\makeatletter
\def\namedlabel#1#2{\begingroup
    #2%
    \def\@currentlabel{#2}%
    \phantomsection\label{#1}\endgroup
}
\makeatother

\newcommand{\ensemblenombre}[1]{\mathbb{#1}}
 \newcommand{\E}{\ensemblenombre{E}}
\newcommand{\Z}{\ensemblenombre{Z}} 
\newcommand{\R}{\ensemblenombre{R}}
\newcommand{\Dim}{d}
\newcommand{\1}{\mathds{1}}
\newcommand{\A}{\mathcal{A}} 
\newcommand{\Intensity}{z}
\newcommand{\Leb}{\mathcal{L}^\Dim}
\newcommand{\Conf}{\omega}
\newcommand{\tConf}{\widetilde{\Conf}}
\newcommand{\ConfSpace}{\Omega}
\newcommand{\MultiConf}{\boldsymbol{\omega}}
\newcommand{\SigmaAlgebra}{\mathcal{F}}
\newcommand{\Poisson}{\pi}
\newcommand{\MultiP}{\boldsymbol{P}}
\newcommand{\MultiPartitionFunction}{\boldsymbol{Z}^{wr}}
\newcommand{\MultiPoisson}{\boldsymbol{\pi}}
\newcommand{\GibbsWR}{\mathcal{G}^{wr}}
\newcommand{\GibbsArea}{\mathcal{G}^{area}}
\newcommand{\Specification}{\mathscr{P}}
\newcommand{\PartitionFunction}{Z^{area}}
\newcommand{\Hamiltonian}{H}
\newcommand{\NccLambda}{N_{cc}^{\Lambda}}
\newcommand{\Ncc}{N_{cc}}
\newcommand{\PartitionFunctionCRCM}{Z^{crc}}

\newcommand{\IntensityThresholdPoisson}{\Intensity_c^p}
\newcommand{\Connected}[1]{\underset{#1}{\longleftrightarrow}}
\newcommand{\DisagreementCoupling}{\mathbb{P}^{\textbf{dc}}}
\newcommand{\IntensityThresholdAreaPercolation}{\widetilde{\Intensity}_c^a}
\newcommand{\IntensityThresholdAreaNonUnicitySym}{\widetilde{\Intensity}_{sym}}
\newcommand{\IntensityThresholdAreaUnPerco}{\bar{\Intensity}_r}
\newcommand{\egras}{\bold{e}}
\newcommand{\tempsarret}{\tau}

\begin{document}
\title{Sharp phase transition for the continuum Widom-Rowlinson model}
\author[1]{David Dereudre}
\author[2]{Pierre Houdebert}
\affil[1]{Laboratoire de Math\'ematiques Paul Painlev\'e\\University of Lille 1, France 
 \texttt{david.dereudre@math.univ-lille1.fr}}
\affil[2]{Aix Marseille Univ, CNRS, Centrale Marseille, I2M, Marseille, France 
 \texttt{pierre.houdebert@gmail.com}}
\maketitle

\begin{abstract}
{The Widom-Rowlinson model (or the Area-interaction model) is a Gibbs point process in $\R^d$ with the formal Hamiltonian 
defined as the volume of $\cup_{x\in\omega} B_1(x)$, where $\omega$ is a locally finite configuration of points and $B_1(x)$ denotes the unit closed ball centred at $x$. 
The model is also tuned by two other parameters: the activity $z>0$ related to the intensity of the process and the inverse temperature $\beta\ge 0$ related to the strength of the interaction. 
In the present paper we investigate the phase transition of the model in the point of view of percolation theory and the liquid-gas transition. First, considering the graph connecting points with distance smaller than $2r>0$, we show that for any $\beta\ge 0$, there exists $0<\IntensityThresholdAreaPercolation(\beta, r)<+\infty$ such that an exponential decay of connectivity at distance $n$ occurs in the subcritical phase (i.e. $z<\IntensityThresholdAreaPercolation(\beta, r)$) and a linear lower bound of the connection at infinity holds in the supercritical case (i.e. $z>\IntensityThresholdAreaPercolation(\beta, r)$). 
These results are in the spirit of recent works using the theory of randomised tree algorithms \cite{ DuminilCopin_Raoufi_Tassion_2019_Sharpness_perco_Voronoi,DuminilCopin_Raoufi_Tassion_2019_Sharpness_perco_RCM, DuminilCopin_Raoufi_Tassion_2018_Sharpness_perco_Boolean_Model}. 
Secondly we study a standard liquid-gas phase transition  related to the uniqueness/non-uniqueness of Gibbs states depending on the parameters $z,\beta$. 
Old results \cite{ruelle_1971, widom_rowlinson} claim that a non-uniqueness regime occurs for $z=\beta$ large enough and it is conjectured that the uniqueness should hold outside such an half line ($z=\beta\ge \beta_c>0$). 
We solve partially this conjecture in any dimension by showing that for $\beta$ large enough the non-uniqueness holds if and only if $z=\beta$. 
We show also that this critical value $z=\beta$ corresponds to the percolation threshold $ \IntensityThresholdAreaPercolation(\beta, r)=\beta$ for $\beta$ large enough, providing a straight connection between these two notions of phase transition.   
} 
  \bigskip

\end{abstract}

\section{Introduction} \label{section_introduction}
The Widom-Rowlinson model (or the Area-interaction model) is a Gibbs point process in $\R^d$ with the formal Hamiltonian given by the volume of the union of balls with radii $1$ and centred at the points of the process;

 $$H(\omega)=\text{Volume}(\cup_{x\in\omega} B_1(x)).$$
 
 By changing the scale, any other value of the radius can be considered as well. Two other parameters, the activity $z>0$, related to the intensity of the process, and the inverse temperature $\beta\ge 0$, related to the strength of the interaction, parametrize the distribution of the process following the standard Boltzmann-Gibbs formalism. In the finite volume regime, the Gibbs measure is absolutely continuous with respect to the Poisson point process with the unnormalized density
\begin{align*}
 f(\omega) \sim z^{\# \omega}e^{-\beta H(\omega)},
\end{align*}
where $\# \omega$ denotes the number of points in $\omega$.
  
The popularity of this model is due to  old results  \cite{chayes_kotecky,ruelle_1971, widom_rowlinson} which prove that the Gibbs measures are not unique in the infinite volume regime for $z=\beta$ large enough. 
This beautiful result is a consequence of a representation of the model via the bi-color Widom-Rowlinson model which identifies the parameters $z$ and $\beta$ as the activities of a two-species model of particles. 
The non-uniqueness of the bi-color Widom-Rowlinson model is proved using Peierls argument, and by symmetry the phase transition is obtained for $z=\beta$ \cite{ruelle_1971}. An alternative proof via the random cluster model and a Fortuin-Kasteleyn representation has been obtain later in \cite{chayes_kotecky}. 
A generalization is proved recently in the case of random unbounded radii \cite{Houdebert_2017_percolation_CRCM}. 
As far as we know, this model and another model with a particular Kac type potential treated in \cite{LMP} are the only models, in the continuum setting without spin, for which a non-uniqueness result is proved. 
Note also that the Area-interaction have been abundantly studied by researchers from different communities in statistical physics, probability theory or spatial statistics  \cite{BVL,Dereudre_Houdebert_2019_JSP_PhaseTransitionWR,  
haggstrom_lieshout_moller,mazel_suhov_stuhl}.

In the present paper we investigate percolation and liquid-gas transition questions for this Area-interaction model. 
These two notions are different but are related and relevant to each other. 
Our first result claims that the Area-interaction model exhibits a sharp phase transition of percolation for the graph connecting points with distance smaller than $2r>0$. Precisely, for any $\beta>0$, there exists a non-trivial threshold $0<\IntensityThresholdAreaPercolation(\beta, r)<+\infty$ such that an exponential decay of connectivity at distance $n$ occurs in the subcritical phase. 
It means that for $z<\IntensityThresholdAreaPercolation(\beta, r)$ the probability (for any Area-interaction model with parameters $z$ and $\beta$) that the point $0$ is connected to the boundary of the sphere of radius $n$ centred at $0$ (i.e. $\partial B_n(0)$), decreases exponentially to zero when $n$ goes to infinity. 
By standard Palm theory arguments, that provides the exponential decay of the size of the clusters in the process itself. 
Moreover a local linear lower bound of the connection at infinity holds in the supercritical case. It means that $z>\IntensityThresholdAreaPercolation(\beta, r)$ not too large, the probability (for any Area-interaction model with parameters $z$ and $\beta$) that the point $0$ is connected to infinity is larger than $c(z-\IntensityThresholdAreaPercolation(\beta, r))$ for a fixed positive constant $c>0$. 
Again, by standard Palm theory arguments, that provides a sub-linear bound for the density of the infinite cluster in the process itself. 
The proofs of these results are in the spirit of recent works using the theory of randomised tree algorithms \cite{ DuminilCopin_Raoufi_Tassion_2019_Sharpness_perco_Voronoi,DuminilCopin_Raoufi_Tassion_2019_Sharpness_perco_RCM, DuminilCopin_Raoufi_Tassion_2018_Sharpness_perco_Boolean_Model}. 
Our main contribution is to prove an OSSS inequality for the Widom-Rowlinson model and to adapt the general strategy of randomised tree algorithms to the setting of interacting continuum particle systems. 
Complementary results show that the function  
$\beta\mapsto \IntensityThresholdAreaPercolation(\beta, r)$ 
is a non-decreasing Lipschitz map. 

Let us now  discuss the sharp liquid-gas phase transition. 
As recalled above, the Area-interaction model exhibits a non-uniqueness regime for $z=\beta$ large enough. This result is called in the literature a liquid-gas phase transition since the pressure is continuous and non-differentiable at the critical point. The derivatives before and after the critical points gives the abrupt difference between the densities of particles in the liquid and gas phases. Since these results in the seventies, it was conjectured that the phase transition is sharp which means that the non-uniqueness occurs if and only if $z=\beta$ larger than a threshold $\beta^*>0$. The results mentioned above do not give any information when $z\neq \beta$, except the standard case where $z$ or $\beta$ are small enough and for which the uniqueness of Gibbs measures is known since long time. In the present paper, we solve partially the conjecture by showing that for $\beta$ large enough (but it is not a threshold) the non-uniqueness holds if and only if $z=\beta$. For moderate values of $\beta$ (not too small and not too large) we obtain uniqueness for $z$ outside the interval $[\IntensityThresholdAreaPercolation(\beta, 1)\;;\;\IntensityThresholdAreaPercolation(., 1)^{-1}(\beta)]$ . 
See Figure \ref{figure_uniqueness} for a precise description of the phase diagram we obtain. 
Our main tool here is an extension of the disagreement percolation argument introduced in \cite{Hofer-temmel_Houdebert_2018} for continuum models. 
Actually we show that the Gibbs measures are unique provided that the wired Area-interaction model does not percolate for $r=1$. 
It means that the Gibbs measures are unique as soon as $z<\IntensityThresholdAreaPercolation(\beta, 1)$ and also for $z>\IntensityThresholdAreaPercolation(., 1)^{-1}(\beta)$ by duality. 
A last result claims that $\IntensityThresholdAreaPercolation(\beta, 1)=\IntensityThresholdAreaPercolation(., 1)^{-1}(\beta)=\beta$ for $\beta$ large enough reducing the interval to the single point $\{\beta\}.$ 

Let us note that the proofs of sharp phase transition, in the settings of percolation theory and liquid-gas transition are quite independent although similar tools and notions are used in both. 
Let us mention also  a similar sharp liquid-gas phase transition obtained for the 2D Widom-Rowlinson model on $\Z^2$ (Theorem 1.1 \cite{HT2004}).
The proof is based on large circuit arguments and depends strongly on the lattice structure. 
It can not be adapted to the continuum setting developed here and moreover it involves only the dimension $d=2$.
  
The paper is organized as follows. In Section \ref{section_preliminaries} we introduce the Area-intercation model and the main required tools (stochastic domination, bicolor Widom-Rowlinson representation, duality). In Section \ref{section_results}, the results are presented and the proofs related to the sharp phase transition of percolation (respectively the liquid-gas phase transition) are given in Section \ref{Section_Perco} (respectively in Section \ref{Section_proof_th_unicite}). An annex Section contains some technical lemmas. 

\section{Preliminaries} \label{section_preliminaries}
\subsection{Space}
Let us consider the state space $\R ^\Dim $ with $\Dim \geq 2$ being the dimension.
Let $\ConfSpace$ be the set of locally finite configurations $\Conf$ on $\R ^\Dim$.
This means that $\#(\Conf \cap \Lambda)< \infty$ for every bounded Borel set $\Lambda$ of $\R^\Dim$, with $\#\Conf$ being the cardinal of the configuration $\Conf$.
We write $\Conf_{\Lambda}$ as a shorthand for $\Conf \cap \Lambda$.
The configuration space is embedded with the usual $\sigma$-algebra $\SigmaAlgebra$ generated by the counting variables.
To a configuration $\Conf \in \ConfSpace$ we associate the germ-grain structure
\begin{align*}
B_r(\Conf) :=
\underset{x \in \Conf}{\bigcup} B_r(x),
\end{align*}
where $B_r(x)$ is the closed ball centred at $x$ with radius $r>0$.

\subsection{Poisson point processes}
Let $\Poisson^{\Intensity}$ be the distribution on $\ConfSpace$ of the homogeneous Poisson point process with intensity $\Intensity >0$. 
Recall that it means 
\begin{itemize}
\item for every bounded Borel set $\Lambda$, the distribution of the number of points in $\Lambda$ under $\Poisson^\Intensity$ is a Poisson distribution of mean $\Intensity \Leb(\Lambda)$, where $\Leb$ stands for the usual $\Dim$-dimensional Lebesgue measure;
\item given the number of points in a bounded $\Lambda$, the points are independent and uniformly distributed in $\Lambda$.
\end{itemize}
We refer to \cite{daley_vere_jones} for details on Poisson point processes.

For $\Lambda \subseteq \R^d$ bounded, we denote by $\Poisson^{\Intensity}_\Lambda$  the restriction of  $\Poisson^{\Intensity}$  on $\Lambda$.
For simplicity the special case of the Poisson point process of unit intensity (i.e. $\Intensity = 1$) is denoted by $\Poisson$, 
and its restriction by $\Poisson_{\Lambda}$.

\subsection{Area-interaction measures}
The area-interaction measures (or the one-color Widom-Rowlinson models) are defined through the standard Gibbs Dobrushin-Lanford-Ruelle formalism prescribing the conditional probabilities.
For a bounded $\Lambda \subseteq \R^\Dim$, we define the \emph{$\Lambda$-Hamiltonian}
\begin{align*}
\Hamiltonian_{\Lambda} (\Conf)
:=
\Leb ( \ B_1(\Conf_{\Lambda}) \setminus B_1(\Conf_{\Lambda^c}) \ ).
\end{align*}
The \emph{area specification} on a bounded $\Lambda \subseteq \R^\Dim$ with boundary condition $\Conf_{\Lambda^c}$ is defined by
\begin{align*}
\Specification_{\Lambda, \Conf_{\Lambda^c}}^{z,\beta}(d \Conf'_\Lambda)
: =
\frac{z^{\#\Conf'_{\Lambda}} \ e^{-\beta \Hamiltonian_{\Lambda}
(\Conf'_{\Lambda} \cup \Conf_{\Lambda^c}) }}
{\PartitionFunction (z, \beta,\Lambda,\Conf_{\Lambda^c})} 
\Poisson_{\Lambda}(d\Conf'_{\Lambda})
\end{align*}
with the standard \emph{partition function}
\begin{align*}
\PartitionFunction (z,\beta ,\Lambda,\Conf_{\Lambda^c})
: =
\int_{\ConfSpace}
z^{\# \Conf'_{\Lambda}} \ e^{-\beta \Hamiltonian_{\Lambda}
(\Conf'_{\Lambda} \cup \Conf_{\Lambda^c}) }
\Poisson_{\Lambda}(d\Conf'_{\Lambda})
\end{align*}
which is always non-degenerate (i.e. $0<\PartitionFunction (z,\beta ,\Lambda,\Conf_{\Lambda^c})<+\infty$).
\begin{remark} \label{remark_additivity_hamiltonian}
There are several possible forms for the Hamiltonian $\Hamiltonian_{\Lambda}$, all of which defining the same specification.
The nice property about our definition of $\Hamiltonian_{\Lambda}$ is the additivity, in the sense that $\Hamiltonian_{\Lambda}(\Conf)$ can be seen as the sum of the contribution of each points, with respect to the already considered ones.

The additivity of the Hamiltonian is one way of ensuring the \emph{compatibility} of the Gibbsian specification, i.e for all $\Lambda \subseteq \Delta \subset R^\Dim$ bounded subsets, for all measurable bounded functions $f$ all boundary condition $\hat{\Conf}$,
\begin{align}
\label{eq_compatibility_specification}
\int_{\ConfSpace} f \ d \Specification^{\Intensity, \beta}_{\Delta,\hat{\Conf}_{\Delta^c}}
=
\int_{\ConfSpace} \int_{\ConfSpace}
f(\Conf'_{\Lambda} \cup \Conf_{\Delta \setminus \Lambda} )
\Specification^{z, \beta}_{\Lambda, 
\Conf_{\Delta \setminus \Lambda} \cup \hat{\Conf}_{\Delta^c}}
(d \Conf'_\Lambda)
\Specification^{\Intensity, \beta}_{\Delta,\hat{\Conf}_{\Delta^c}}(d\Conf_\Delta).
\end{align}

\end{remark}
\begin{definition}
\label{def_area_model}
A probability measure $P$ on $\ConfSpace$ is an \emph{area-interaction measure} of activity $z$ and inverse temperature $\beta$,  written 
$P \in \GibbsArea_{z, \beta}$, if for every bounded Borel set $\Lambda \subseteq \R^d$ and  every bounded measurable function $f$,
\begin{align}
\label{eq_dlr_area_model}
\int_{\ConfSpace} f \ d P 
=
\int_{\ConfSpace} \int_{\ConfSpace}
f(\Conf'_{\Lambda} \cup \Conf_{\Lambda^c} )
\Specification^{z, \beta}_{\Lambda, \Conf_{\Lambda^c}}(d \Conf'_\Lambda)
P (d \Conf ).
\end{align}
The equations \eqref{eq_dlr_area_model}, for all bounded $\Lambda$, are called \emph{DLR equations}, after Dobrushin, Lanford and Ruelle.
Those equations prescribe the conditional probabilities of a Gibbs measure.
\end{definition}
\subsection{Stochastic domination}
Let us discuss  stochastic domination, which is going to be a key element of several proofs of the paper.
Recall that an event $E \in \SigmaAlgebra$ is said \emph{increasing} if for $\Conf' \in E$ and $\Conf \supseteq \Conf'$, we have $\Conf \in E$.
 This definition naturally extend to define \emph{increasing functions} 
$f:\ConfSpace \to \R$.
Finally if $P$ and $P'$ are two probability measures on $\ConfSpace$, the measure $P$ is said to \emph{stochastically dominate} the measure $P'$, written $P' \preceq P$, if $P'(E) \leq P(E)$ for every increasing event $E \in \SigmaAlgebra$.

The following proposition is a direct application of the classical Georgii and K\"uneth stochastic domination result
\cite[Theorem 1.1]{georgii_kuneth} and gives standard stochastic dominations.
\begin{proposition}
\label{propo_dom_sto_poisson_area}
For every bounded $\Lambda \subseteq \R^d$,
\begin{itemize}
\item for every boundary condition $\Conf_{\Lambda^c}$ and every $\Intensity, \beta$ we have
\begin{align*}
\Poisson^{\Intensity e^{-\beta v_\Dim}}_{\Lambda}
\preceq
\Specification_{\Lambda, \Conf_{\Lambda^c}}^{\Intensity, \beta }(d \Conf'_\Lambda)
\preceq 
\Poisson^{\Intensity}_{\Lambda},
\end{align*}
where $v_\Dim$ is the volume of the unit ball in dimension $\Dim$.

This implies in particular that every $P \in \GibbsArea_{\Intensity, \beta}$ satisfies
\begin{align*}
\Poisson^{\Intensity e^{-\beta v_\Dim}}
\preceq
P  
\preceq 
\Poisson^{\Intensity}.
\end{align*}
\item For every boundary conditions
$\Conf^{1}_{\Lambda^c} \subseteq \Conf^{2}_{\Lambda^c}$, every $\Intensity_1 \leq \Intensity_2$ and every $\beta_1 \geq \beta_2$ we have
\begin{align*}
\Specification_{\Lambda, \Conf^1_{\Lambda^c}}^{\Intensity_1, \beta_1 }(d \Conf'_\Lambda)
\preceq 
\Specification_{\Lambda, \Conf^2_{\Lambda^c}}^{\Intensity_2, \beta_2}(d \Conf'_\Lambda) .
\end{align*}
\end{itemize}
\end{proposition}
\subsection{Free and wired measures}
Two particular area-interaction measures are constructed as follows.
Consider the increasing sequence $\Lambda_n := ]-n,n]^\Dim$ and consider the \emph{free and wired area-interaction measures}  on the bounded box $\Lambda_n$, denoted by $P^{\Intensity, \beta}_{n,free}$ and $P^{\Intensity,\beta}_{n,wired}$ and defined as
\begin{align*}
P^{\Intensity, \beta}_{n,free} \left(d \Conf'_{\Lambda_n}\right) 
&:= 
\frac{z^{\# \Conf'_{\Lambda_n}} \ 
e^{-\beta \Leb\left(\ B_1 \left(\Conf'_{\Lambda_n} \right) \ \right) }}
{\PartitionFunction(z, \beta,n,free)}  \
\Poisson_{\Lambda_n}(d\Conf'_{\Lambda_n});
\\ 
P^{\Intensity, \beta}_{n,wired} \left(d \Conf'_{\Lambda_n} \right) 
& := 
\frac{z^{\# \Conf'_{\Lambda_n}} \ 
e^{-\beta \Leb \left(
\ B_1\left(\Conf'_{\Lambda_n} \right) \cap \Lambda_{n-1} \ \right) }}
{\PartitionFunction(z, \beta,n,wired)} \
\Poisson_{\Lambda_n}(d\Conf'_{\Lambda_n});
\end{align*}
where $\PartitionFunction(z, \beta,n,free)$ and
$\PartitionFunction(z, \beta,n,wired)$ are the normalising constants.
The measure $P^{\Intensity, \beta}_{n,free} $ is simply 
$\Specification_{\Lambda_n, \emptyset}^{\Intensity, \beta}$, whereas 
$P^{\Intensity, \beta}_{n,wired}$ is the limiting case where the boundary condition would be filled with  points on the boundary of $\Lambda_n$.

From \cite[Theorem 1.1]{georgii_kuneth} we get the following proposition.

\begin{proposition}
\label{propo_dom_sto_free_wired_boite}
For every $n$ and every $\Intensity, \beta$ we have
\begin{itemize}
\item $P^{\Intensity, \beta}_{n,free} \preceq P^{\Intensity, \beta}_{n+1,free}$;
\item ${P^{\Intensity, \beta}_{n+1,wired}}_{|\Lambda_n} \preceq P^{\Intensity, \beta}_{n,wired}$,
\end{itemize}
where ${P^{\Intensity, \beta}_{n+1,wired}}_{|\Lambda_n}$ stands for the measure 
$P^{\Intensity, \beta}_{n+1,wired}$ restricted to $\Lambda_n$.
\end{proposition}

From Proposition \ref{propo_dom_sto_free_wired_boite} and using Carathéodory's extension theorem we get the existence of $P^{\Intensity, \beta}_{free}$ and $P^{\Intensity, \beta}_{wired}$.
Those probability measures are, thanks to \cite[Theorem 11.1.VII]{daley_vere_jones}, weak limits of the sequences 
$\left(P^{\Intensity, \beta}_{n,free}\right)_n$ and 
$\left(P^{\Intensity, \beta}_{n,wired}\right)_n$.
They are also stationary (see \cite{chayes_kotecky} for details).

\begin{proposition}
\label{propo_domination_sandwich}
For every $\Intensity_1 \leq \Intensity_2$ and $\beta_1 \geq \beta_2$,
\begin{itemize}
\item $P^{\Intensity_1, \beta_1}_{free} \in \GibbsArea_{\Intensity_1,\beta_1}$
and $P^{\Intensity_1, \beta_1}_{wired} \in \GibbsArea_{\Intensity_1,\beta_1}$;
\item $P^{\Intensity_1, \beta_1}_{free} \preceq P^{\Intensity_2, \beta_2}_{free}$
and $P^{\Intensity_1, \beta_1}_{wired} \preceq P^{\Intensity_2, \beta_2}_{wired}$;
\item 
$P^{\Intensity_1, \beta_1}_{free} \preceq P 
\preceq P^{\Intensity_1, \beta_1}_{wired}$ \quad for all 
$P \in \GibbsArea_{\Intensity_1,\beta_1}$.
\end{itemize}
\end{proposition}
As a consequence of the first item of Proposition \ref{propo_domination_sandwich}, we know that the set of area-interaction measures $\GibbsArea_{\Intensity,\beta}$ is never empty.
From the last item of Proposition \ref{propo_domination_sandwich}, the question of uniqueness of the area-interaction measure translates to the question of the equality of measures
$P^{\Intensity_1, \beta_1}_{free} = P^{\Intensity_1, \beta_1}_{wired}$.
The next Proposition is stating that this equality happens for a lot of parameters $(\Intensity, \beta)$.
\begin{proposition}
\label{proposition_non_unicite_au_plus_denombrable}
For all $\beta>0$, the set
$ \{ 
\Intensity >0 , P^{\Intensity, \beta}_{free} 
\not =  P^{\Intensity, \beta}_{wired}
\} $
is at most countable.
\end{proposition}
The proof of this proposition is related to standard differentiability/convexity arguments of the pressure function. 
See for instance Theorem 3.34 in \cite{Velenik} for a proof in the case of Ising model or Theorem 4.2 in \cite{HT2004} for the lattice Widom-Rowlinson model. A direct adaptation for the continuum area-interaction measure is achievable and omitted here for brevity.

\subsection{Bicolor Widom-Rowlinson representation of area-interaction measures}
The bicolor Widom-Rowlinson model is simply defined as the reunion of two Poisson Boolean models (with deterministic radii equal to 0.5) conditioned on a hard-core non overlapping condition between the two Boolean models.
A formal definition using standard DLR formalism is given below.


\begin{definition}
\label{defi_wr}
Let $\MultiConf:= (\Conf^1,\Conf^2)$ denotes a couple of configurations.
Let $\A:=\{ (\Conf^1,\Conf^2) \in \ConfSpace^2, 
B_{1/2}(\Conf^1)\cap B_{1/2}(\Conf^2) = \emptyset \}$ be the event of authorised (couple of) configurations.
Let $\MultiPoisson^{\Intensity_1,\Intensity_2}:= 
\Poisson^{\Intensity_1} \otimes \Poisson^{\Intensity_2}$.

Then a probability measure $\MultiP$ on $\ConfSpace^2$ is a \emph{Widom-Rowlinson measure} with parameters $\Intensity_1,\Intensity_2$, written
$\MultiP \in \GibbsWR_{\Intensity_1,\Intensity_2}$, if $\MultiP(\A)=1$ and if for every bounded $\Lambda \subseteq \R^\Dim$ and every bounded measurable function $f$,
\begin{align}
\label{eq_dlr_WR}
\int_{\ConfSpace^2} f  d \MultiP 
=
\int_{\ConfSpace^2} \int_{\ConfSpace^2}
f(\MultiConf'_{\Lambda} \cup \MultiConf_{\Lambda^c} )
\frac{\1_{\A}( \MultiConf'_{\Lambda} \cup \MultiConf_{\Lambda^c})}
{\MultiPartitionFunction 
(\Lambda, \Intensity_1,\Intensity_2, \MultiConf_{\Lambda^c})}
\MultiPoisson_{\Lambda}^{\Intensity_1,\Intensity_2} (d \MultiConf'_{\Lambda})
\MultiP (d \MultiConf),
\end{align}
with $\MultiPartitionFunction 
(\Lambda, \Intensity_1,\Intensity_2, \MultiConf_{\Lambda^c})$ being the standard partition function associated to the Widom-Rowlinson interaction.
\end{definition}

The following proposition is the standard relation between one-color and bi-color Widom-Rowlinson models.

\begin{proposition}
\label{propo_representation_wr_area}
${}$
\begin{itemize}
\item Let $\MultiP \in \GibbsWR_{\Intensity_1,\Intensity_2}$.
Then the first marginal of $\MultiP$ is an area-interaction measure of activity $\Intensity = \Intensity_1$ and inverse temperature $\beta= \Intensity_2$.
The analogue is true for the second marginal.
\item Let $P \in \GibbsArea_{\Intensity,\beta}$.
Then the measure $\MultiP := 
\Poisson^{\beta}_{\R^\Dim \setminus B_1(\Conf^1)}(d\Conf^2) P(d\Conf^1)$ is a Widom-Rowlinson measure: 
$\MultiP \in \GibbsWR (\Intensity_1=\Intensity, \Intensity_2=\beta)$.
\end{itemize}
As a consequence of these two points, the sets $\GibbsArea_{\Intensity,\beta}$ and
$\GibbsWR (\Intensity_1=\Intensity, \Intensity_2=\beta)$ are in bijection. 
By symmetry of $\GibbsWR (\Intensity_1, \Intensity_2)$ with respect to $\Intensity_1$, $\Intensity_2$, the sets $\GibbsArea_{\Intensity,\beta}$ and $\GibbsArea_{\beta, \Intensity}$ are in bijection as well; 
this property is called the duality property.
Only this duality property will be used later, in the proof of Theorem \ref{theo_unicite_avant_threshold}.
\end{proposition}

The proof relies on the fact that integrating
the density of the bicolor Widom-Rowlinson model with respect to one marginal, one gets 
the density of the monocoloured Widom-Rowlinson model.
We omit the proof and we refer to 
\cite{chayes_kotecky, georgii_haggstrom_maes,widom_rowlinson} for details.

\subsection{Percolation}
The theory of percolation studies the connectivity in random structures.
Formally the percolation is defined as follows.
\begin{definition}
Let $r>0$;
\begin{itemize}
\item
two sets $\Lambda_1,\Lambda_2 \subseteq \R^\Dim$ are said to be $r$-connected in $\Conf$, written $\Lambda_1 \Connected{B_r(\Conf)}\Lambda_2$ (or $\Lambda_1 \Connected{r}\Lambda_2$ when there is no possible confusion) if 
$B_r(\Conf)\cup \Lambda_1 \cup \Lambda_2$ has a connected component overlapping both $\Lambda_1$ and $\Lambda_2$;
\item
a configuration $\Conf$ is said to $r$-percolate if the germ-grain structure 
$B_r(\Conf)$ has at least one unbounded connected component;
\item
a probability measure $P$ on $\ConfSpace$ is said to \emph{$r$-percolate} (respectively \emph{do not percolate}) if 
$P(\{ \Conf  \text{ $r$-percolates} \})=1$ 
(respectively $P(\{ \Conf \text{ $r$-percolates} \})=0$).
\end{itemize}
\end{definition}

In the next proposition we state the standard percolation phase transition of the Poisson Boolean model. See for instance  \cite{DuminilCopin_Raoufi_Tassion_2018_Sharpness_perco_Boolean_Model} for a modern proof.
\begin{proposition}
\label{propo_perco_poisson}
For every $r>0$, there exists $0<\IntensityThresholdPoisson(r)<\infty$, called \emph{$r$-percolation threshold} of the Poisson Boolean model, such that
\begin{itemize}
\item for every $\Intensity<\IntensityThresholdPoisson (r)$, the measure $\Poisson^{\Intensity}$ does not $r$-percolate, and we have the existence of $c:=c(r,\Intensity)>0$ such that
\begin{align}
\label{eq_decroissance_expo_poisson}
\Poisson^{\Intensity}\left( 0 \Connected{r} \partial \Lambda_n \right)
\leq  \exp (-cn),
\end{align}
where $\partial \Lambda_n$ is the boundary of the set $\Lambda_n=]-n,n]^\Dim$;
\item For every $\Intensity>\IntensityThresholdPoisson (r)$, the measure $\Poisson^{\Intensity}$ $r$-percolates, and we have the existence of $c':=c'(r)>0$ such that for $z$ in a neighbourhood of $\IntensityThresholdPoisson (r)$
\begin{align*}
\Poisson^{\Intensity}\left( 0 \Connected{r} \infty \right)
\geq c'(\Intensity - \IntensityThresholdPoisson (r)).
\end{align*}
\end{itemize}
\end{proposition}
Concerning area-interaction measures such a behaviour is not proven, and is one of the questions investigated in this paper.
But as a consequence of the Propositions \ref{propo_dom_sto_poisson_area},
\ref{propo_domination_sandwich},
\ref{proposition_non_unicite_au_plus_denombrable} and
\ref{propo_perco_poisson}
we have the existence of a non degenerate percolation threshold, common to all area-interaction measures.
This is stated in the following Proposition.

\begin{proposition}
\label{propo_threshold_area}
For all $\beta>0$ and $r>0$, there exists 
$0<\IntensityThresholdAreaPercolation(\beta,r) <\infty$ such that
\begin{itemize}
\item
for all $\Intensity < \IntensityThresholdAreaPercolation(\beta,r)$, 
any area-interaction measure $P \in \GibbsArea_{\Intensity, \beta}$
almost never $r$-percolates, i.e
\begin{align*}
P(\{ \Conf \ r\text{-percolates}\})
= 0;
\end{align*}
\item
for all $\Intensity > \IntensityThresholdAreaPercolation(\beta,r)$, any area-interaction measures $P \in \GibbsArea_{\Intensity, \beta}$ almost surely $r$-percolates, i.e
\begin{align*}
P(\{ \Conf \ r\text{-percolates}\})
= 1.
\end{align*}
\end{itemize}
\end{proposition}
\begin{proof}
The fact that both the free and wired measures have the same threshold is a consequence of Proposition \ref{propo_domination_sandwich} and Proposition \ref{proposition_non_unicite_au_plus_denombrable}.
The non-degeneracy of this threshold follows from the non-degeneracy of the Poisson Boolean model percolation threshold,  which is a standard property included in the statement of Proposition \ref{propo_perco_poisson},
and from the dominations of
Proposition \ref{propo_dom_sto_poisson_area}:
for all $P \in \GibbsArea_{\Intensity, \beta}$
\begin{equation}\label{encadrement}
\Poisson^{\Intensity e^{-\beta v_d}}
\preceq  P  \preceq
\Poisson^{\Intensity}.
\end{equation}
\end{proof}
\begin{remark}
Let us first notice that by a scaling argument, the percolation thresholds of the Poisson Boolean model satisfies the well-known relation $\IntensityThresholdPoisson (r)
=
\frac{1}{r^\Dim}
\IntensityThresholdPoisson (1)$. Then, from the stochastic dominations \eqref{encadrement}, we have the following bound on the percolation threshold of the area-interaction measures: for all $r>0$ and $\beta\geq 0$,
\begin{align*}
\IntensityThresholdPoisson(r)
\ \leq \
\IntensityThresholdAreaPercolation(\beta,r)
\ \leq \
\IntensityThresholdPoisson(r) \exp (\beta v_\Dim).
\end{align*}

\end{remark}

\section{Results} \label{section_results}
Let us now present our results related to the phase transition of the area-interaction measures. The proofs are given in the following sections. 

\subsection{Sharp phase transition of percolation}
The first result  proves a \emph{sharp phase transition} of percolation for the area-interaction measures in the spirit of Proposition \ref{propo_perco_poisson} for the Boolean model. That means exponential decay of connectivity at distance $n$ in the subcritical phase and a local linear lower bound of the connection at infinity in the supercritical case. 

\begin{theorem}
\label{theo_sharpness_perco_area}
Let $\beta \geq 0$.
\begin{enumerate}
\item 
For all $\Intensity < \IntensityThresholdAreaPercolation(\beta, r) $, there exists 
$\alpha_1 =\alpha_1 (\Intensity,\beta, \Dim, r)>0$ such that for all $P \in \GibbsArea_{\Intensity, \beta}$ and all $n$,
\begin{align}
\label{eq_perco_area_decroissance_exp_sous_critique}
P \left( 0 \Connected{r} \partial \Lambda_n \right)
\leq
\exp (- \alpha_1 n).
\end{align}
\item
There exists $\alpha_2= \alpha_2 (\beta, \Dim, r)$ such that for all 
$\Intensity > \IntensityThresholdAreaPercolation(\beta, r)$ small enough and all $P \in \GibbsArea_{\Intensity, \beta}$,
\begin{align}
\label{eq_perco_area_surlineaire_sur_critique}
P \left( 0 \Connected{r} \infty \right)
\geq
\alpha_2 (\Intensity -\IntensityThresholdAreaPercolation(\beta,r)).
\end{align}
\end{enumerate}
\end{theorem}
The proof of this theorem relies on the theory of randomised algorithms developed by Duminil-Copin, Raoufi and Tassion in a series of papers
\cite{DuminilCopin_Raoufi_Tassion_2019_Sharpness_perco_Voronoi, DuminilCopin_Raoufi_Tassion_2019_Sharpness_perco_RCM,
DuminilCopin_Raoufi_Tassion_2018_Sharpness_perco_Boolean_Model}.
The main ingredient, and our main contribution with respect to what was already done, is the proof of an OSSS-type inequality which gives a control of the variance of a function $f$ by a bound depending on the influence of each point of the process.
The proof of this inequality relies on a procedure, sampling an area-interaction configuration using a dominating Poisson configuration.
This procedure is in some sense monotonic with respect to the dominating Poisson configuration. The proof is given in Section \ref{Section_Perco}.

The next proposition gives some qualitative properties of the function $\beta\mapsto \IntensityThresholdAreaPercolation(\beta, r)$ and exact values for $\beta$ large enough. 
The proof is given in  Section \ref{Section_Perco} as well.

\begin{proposition}\label{proposition_property_threshold}
For every $r>0$, the function 
$\beta\mapsto \IntensityThresholdAreaPercolation(\beta, r)$ 
is a non-decreasing Lipschitz map from $\R^+$ to $[\IntensityThresholdPoisson(r),+\infty)$. In particular, it is continuous. Moreover for every $r>0$, there exists $0<\tilde{\beta}_r<\infty$ such that for $\beta>\tilde{\beta}_r$, the equality 
$\IntensityThresholdAreaPercolation(\beta,r) = \beta$ holds.
\end{proposition}

\subsection{Sharp liquid-gas phase transition}
The other question of interest is the \emph{Sharp liquid-gas phase transition} for which there are several definitions based either on the regularity of the pressure or the uniqueness/non uniqueness of Gibbs measures. Here we say that a \emph{ sharp liquid-gas phase transition} occurs at temperature $1/\beta$ if there exists only one value $z$ such that the Gibbs measures are not unique. This phenomenon is conjectured for several models but there does no exist complete rigorous proof in the continuum. Here we improve existing results for the area-interaction measures. 

\subsubsection{Already known results}
Several results are already known on this subject.
First, it is well-known that the set of Gibbs measures is generally reduced to a singleton when the parameters $z$ or/and $\beta$ are small enough (see for instance \cite{ruelle_livre_1969}). 
As a consequence of a recent disagreement percolation result \cite{Hofer-temmel_Houdebert_2018}, explicit bounds related to $1$-percolation threshold of the Poisson Boolean model are given.

\begin{proposition}
\label{propo_unicite_disagreement_perco}
Recall that $\IntensityThresholdPoisson (r)$ is the percolation threshold of the Poisson Boolean model of constant radii $r$.
Then for every $\Intensity <\IntensityThresholdPoisson (1)$ and every $\beta \geq 0$, there is an unique area-interaction measure. Moreover, by duality, for every $\beta <\IntensityThresholdPoisson (1)$ and every $\Intensity \geq 0$, the uniqueness occurs as well.
\end{proposition}
In addition, a Fortuin-Kasteleyn representation and  percolation properties of the Continuum Random Cluster Model allow to prove a non uniqueness result for the symmetric bicolor Widom-Rowlinson model \cite{chayes_kotecky}. This result translates directly, thanks to Proposition \ref{propo_representation_wr_area}, to a non uniqueness result of the area-interaction measure.
\begin{proposition}
\label{propo_phase_transition_homogeneous_fk}
There exists $0<\IntensityThresholdAreaNonUnicitySym<\infty$ such that for all 
$\Intensity > \IntensityThresholdAreaNonUnicitySym$
\begin{itemize}
\item
the measure $P^{\Intensity, \Intensity}_{wired}$ does $1/2$-percolate;
\item
the measure $P^{\Intensity, \Intensity}_{free}$ does not $1/2$-percolate;
\end{itemize}
hence
we have $P^{\Intensity, \Intensity}_{free} \not =P^{\Intensity, \Intensity}_{wired}$, and therefore $\# \GibbsArea(\Intensity, \Intensity)>1$.
\end{proposition}
So in the symmetric case $z=\beta$, a standard phase transition is already known, where uniqueness is obtain at low activity $z$ and non-uniqueness at large activity. 
However it is not proved that there exists a threshold between both regimes. As far as we know, this conjecture is still open today. 

In the  non-symmetric case $z\neq\beta$, very few is known expect from Proposition \ref{propo_unicite_disagreement_perco}. In particular, the sharp phase transition around the symmetric case $z=\beta$ was unknown.

\subsubsection{New results about uniqueness}
It is conjectured that the non-uniqueness holds if and only if $z=\beta$ larger than a certain threshold $\beta^*>0$. We do not solve this conjecture here but we show in Corollary \ref{corollary_sharp_liquid_gas_phase_transition} that for $\beta$ large enough the non-uniqueness holds only for $z=\beta$. 
Actually we succeed to prove uniqueness in a larger domain drawn in Figure \ref{figure_uniqueness}.

Our main Theorem, given below, ensures the uniqueness as soon as the area-measures do not $1$-percolate.  
\begin{theorem}
\label{theo_unicite_avant_threshold}
For all $\beta \geq 0$ and 
$\Intensity< \IntensityThresholdAreaPercolation(\beta,1)$, we have
$P^{\Intensity, \beta}_{free}
=
P^{\Intensity, \beta}_{wired}$, and therefore there is uniqueness of the area-interacton measure. By duality the result holds also for all $z \geq 0$ and $\beta<\IntensityThresholdAreaPercolation(z,1)$.
\end{theorem}
\begin{figure}[h!]
\centering
\begin{tikzpicture}[scale=0.7]
\draw [>=latex,->] (0,0) -- (10,0) node[right] {$\Intensity$};
\draw [>=latex,->] (0,0) -- (0,10) node[above] {$\beta$};
\draw (2,0) node[below]{$\IntensityThresholdPoisson(1)$} node {$\bullet$};
\draw (0,2) node[left]{$\IntensityThresholdPoisson(1)$} node {$\bullet$};
\draw [dashed] (0,0) -- (10,10) node[above right] {$\Intensity = \beta$};
\draw (5,5) node {$\bullet$};
\draw (5,5) node[below right] {$\IntensityThresholdAreaNonUnicitySym$};
\draw [red] (5,5) -- (10,10);
\draw [>=latex,<-] (7.5,7.5) -- (11,9.5); 
\node[draw,align=center] at (14,8.5) {Non uniqueness \\ symmetric case \\ Prop \ref{propo_phase_transition_homogeneous_fk}};
\draw [line width=1pt] (2,0) arc (180:120.05:8)--(7.5,7.725);
\draw [line width=1pt] (7.5,7.725) arc (300:315:5)--(10,10);

\node[align=center] at (-3,4) 
{$\beta \mapsto \IntensityThresholdAreaPercolation (\beta,1)$};
\draw [>=latex,<-] (2.7,3.5) -- (-0.5,4);

\fill [blue, opacity=0.60] (0,0)--(2,0)-- (2,0) arc (180:120.05:8)--
(7.5,7.725)--(7.5,7.725) arc (300:315:5)-- (10,10)--(0,10);

\fill [blue, opacity=0.60](0,0)-- (0,2)-- (0,2) arc (270:329.5:8)--
(7.725,7.5)--(7.725,7.5) arc (150:135:5)-- (10,10)--(10,0)
;

\node[draw,align=center] at (5,11.5) {uniqueness \\ Th \ref{theo_unicite_avant_threshold}};
\draw [>=latex,<-] (3,7.5) -- (5,10.5);
\draw [>=latex,<-] (7,3) -- (5,10.5);
\draw (4,3) node[above] {\bf \Huge ?} ;
\end{tikzpicture}
\caption{Uniqueness/non-uniqueness regimes for the area-interaction measures with parameters $z,\beta$.}
\label{figure_uniqueness}
\end{figure}
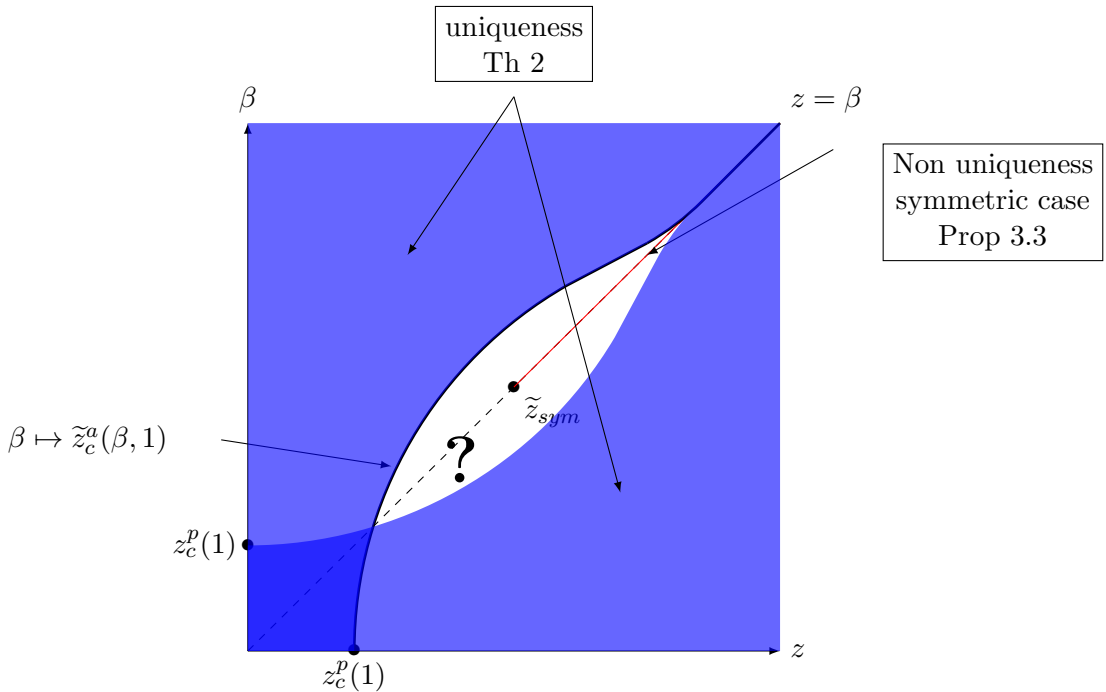
The proof of this theorem relies on a generalization of the disagreement percolation technique, relying on the construction of a coupling, called \emph{disagreement coupling} comparing the influence of the boundary condition to a domination Poisson point process.
Using the monotonicity of the area interaction, see Proposition \ref{propo_dom_sto_poisson_area}, a better dominating measure is the wired area-interaction measure.
The dominating measure is not a Poisson point process and therefore the construction of the disagreement coupling is more elaborate, even though it still relies on the original idea of van den Berg and Maes \cite{vandenberg_maes}.
The proof of Theorem \ref{theo_unicite_avant_threshold} is done in Section \ref{Section_proof_th_unicite}.
 Although both Theorem \ref{theo_sharpness_perco_area} and \ref{theo_unicite_avant_threshold} are related, there proofs are independent from one another. 
\begin{corollary}\label{corollary_sharp_liquid_gas_phase_transition}
For $\beta$ larger than $\tilde{\beta}_1$, the area-interaction measures with parameters $z,\beta$ are non-unique if and only if $z=\beta$. The sharp liquid-gas transition occurs.
\end{corollary}
\begin{proof}
It is a direct consequence of Theorem \ref{theo_unicite_avant_threshold}, Proposition \ref{proposition_property_threshold}, 
Proposition \ref{propo_phase_transition_homogeneous_fk} and the duality property given in Proposition \ref{propo_representation_wr_area}.
\end{proof}

\section{Proofs related to percolation results}\label{Section_Perco}
In this section we give the proofs of Theorem \ref{theo_sharpness_perco_area} and Proposition \ref{proposition_property_threshold} involving the sharp phase transition of percolation.

\subsection{Proof of Theorem \ref{theo_sharpness_perco_area}}
First let us note that it is enough to prove Theorem \ref{theo_sharpness_perco_area} for the wired area-interaction measure $P_{wired}^{\Intensity, \beta}$. Indeed, recall that from Proposition \ref{propo_domination_sandwich} we have the sandwich domination: for all $P \in \GibbsArea_{\Intensity, \beta}$,
$P_{free}^{\Intensity, \beta}
\preceq  P  \preceq
P_{wired}^{\Intensity, \beta} $.
Therefore the equation \eqref{eq_perco_area_decroissance_exp_sous_critique}, i.e. the exponential decay of connectivity when 
$\Intensity < \IntensityThresholdAreaPercolation (\beta, r)$, translates directly from $P_{wired}^{\Intensity, \beta} $ to all $P \in \GibbsArea_{\Intensity,\beta}$.
For the equation \eqref{eq_perco_area_surlineaire_sur_critique}, consider 
$\Intensity >\Intensity'> \IntensityThresholdAreaPercolation(\beta,r)$ such that
$P_{free}^{\Intensity' ,\beta}=P_{wired}^{\Intensity' ,\beta}$.
From Proposition \ref{proposition_non_unicite_au_plus_denombrable} the parameter $\Intensity'$ can be considered as close to $\Intensity$ as we need.
But once again from Proposition \ref{propo_domination_sandwich} we have
\begin{align*}
P\left( 0 \Connected{r} \infty  \right) 
& \geq
P_{free}^{\Intensity, \beta}\left( 0 \Connected{r} \infty  \right)
\\ &\geq
P_{free}^{\Intensity', \beta}\left( 0 \Connected{r} \infty  \right) 
\\ & = 
P_{wired}^{\Intensity', \beta}\left( 0 \Connected{r} \infty  \right)
\\ & \geq \alpha_2 (\Intensity' - \IntensityThresholdAreaPercolation )
\underset{\Intensity' \to \Intensity}{\longrightarrow}
\alpha_2 (\Intensity - \IntensityThresholdAreaPercolation ),
\end{align*}
and Equation \eqref{eq_perco_area_surlineaire_sur_critique} is proved.

Through the remainder of this section the parameters $\beta>0$ and $r>0$ are fixed and might be omitted from notations and we consider only the wired case.
Let $\mu_n^\Intensity:=P^{\Intensity, \beta}_{3n+2, wired}$.
We are considering the connection probability
\begin{align*}
\theta_n(\Intensity)
=
\mu_n^\Intensity (0 \Connected{r} \partial \Lambda_n),
\end{align*}
where $\partial \Lambda_n$ is the boundary of $\Lambda_n=]-n,n]^\Dim$.
\begin{remark}
The term $3n+2$ was chosen for several reasons.
First the term "$+2$" is there to ensure that the wired measure $\mu_n^\Intensity$ is well defined, even for $n=0$.
Second the factor $3$ is there to ensure a good inclusion of boxes in \eqref{eq_preuve_sharpness_majoration_revealment}.
\end{remark}
\begin{lemma}
\label{lemme_limite_theta_n}
For each $z$, the sequence $(\theta_n(\Intensity))$ converges and we have
\begin{align*}
\theta (z) 
:= 
\lim_{n \to \infty} \theta_n(\Intensity)
\leq
P^{\Intensity, \beta}_{wired}\left(0 \Connected{r} \infty \right).
\end{align*}
\end{lemma}
\begin{proof}
The event $\{0 \Connected{r} \partial \Lambda_n \}$ depends only on the points inside $\Lambda_{n+r}$, which is included in $\Lambda_{3n+2}$ as soon as $n\geq r/2$.
Therefore, using Proposition \ref{propo_dom_sto_free_wired_boite} we have for such $n$:
\begin{align*}
\theta_n(\Intensity)
& =
P^{\Intensity, \beta}_{3n+2, wired} (0 \Connected{r} \partial \Lambda_n)
\\ & \geq 
P^{\Intensity, \beta}_{3(n+1)+2, wired} (0 \Connected{r} \partial \Lambda_n)
\\ & \geq 
P^{\Intensity, \beta}_{3(n+1)+2, wired} (0 \Connected{r} \partial \Lambda_{n+1})
=
\theta_{n+1}(\Intensity).
\end{align*}
Hence the sequence is decreasing for $n$ large enough and the convergence follows.
For the inequality, just notice that for any $k$
\begin{align*}
\theta(\Intensity)
& = 
\lim_{n \to \infty} \theta_n(\Intensity)
\leq
\lim_{n \to \infty}
P^{\Intensity,\beta}_{3n+2,wired} (0 \Connected{r} \partial \Lambda_k)
=
P^{\Intensity,\beta}_{wired} (0 \Connected{r} \partial \Lambda_k).
\end{align*}
Letting $k$ go to infinity yields the result.
\end{proof}
Recall that $\IntensityThresholdAreaPercolation (\beta,  r) \in ]0,\infty[$ is the percolation threshold of the wired (and free as well) area-interaction measure, defined in Proposition \ref{propo_threshold_area}.
\begin{theorem}
\label{theo_sharpness_perco_mesure_wired}
\begin{enumerate} Let $\beta>0$ and $r>0$.
\item 
For all $\Intensity < \IntensityThresholdAreaPercolation(\beta,r) $, there exists 
$\alpha_1 =\alpha_1 (\Intensity,\beta, \Dim, r)>0$ such that, for all $n$,
\begin{align}
\label{eq_perco_wired_decroissance_exp_sous_critique}
\theta_n(\Intensity)
\leq
\exp (- \alpha_1 n).
\end{align}
\item
There exists $\alpha_2= \alpha_2 (\beta, \Dim, r)$ such that for all 
$\Intensity > \IntensityThresholdAreaPercolation(\beta,r) $ small enough,
\begin{align}
\label{eq_perco_wired_surlineaire_sur_critique}
\theta (\Intensity)
\geq
\alpha_2 (\Intensity - \IntensityThresholdAreaPercolation ).
\end{align}
\end{enumerate}
\end{theorem}

Before proving this theorem, let us see quickly how it leads to the proof of Theorem \ref{theo_sharpness_perco_area}.
The equation \eqref{eq_perco_wired_surlineaire_sur_critique} together with Lemma \ref{lemme_limite_theta_n} implies the equation \eqref{eq_perco_area_surlineaire_sur_critique}, while \eqref{eq_perco_area_decroissance_exp_sous_critique} is a consequence of \eqref{eq_perco_wired_decroissance_exp_sous_critique} and Proposition \ref{propo_dom_sto_free_wired_boite}.

The proof of Theorem \ref{theo_sharpness_perco_mesure_wired} relies on the theory of randomised algorithms popularized in
 \cite{DuminilCopin_Raoufi_Tassion_2019_Sharpness_perco_Voronoi,DuminilCopin_Raoufi_Tassion_2019_Sharpness_perco_RCM,
 DuminilCopin_Raoufi_Tassion_2018_Sharpness_perco_Boolean_Model}.
 First we are going to prove in Section \ref{section_OSSS} a generalized version of the OSSS inequality satisfied by $\Specification^{\Intensity,\beta}_{\Lambda,\hat{\Conf}_{\Lambda^c}}$ for every cube $\Lambda$ and every boundary condition $\hat{\Conf}_{\Lambda^c}$.
 
This inequality, valid in particular for the wired measure $P^{\Intensity,\beta}_{3n+2,wired}$, is then applied in Section \ref{section_utilisation_OSSS} in order to get the result.
\subsection{OSSS inequality}\label{section_OSSS}
\subsubsection{Introduction of the formalism}
In this section $\Lambda$ is a fixed cube of side length $c>0$.
Let $\varepsilon>0$ such that $c / \varepsilon$ is a positive integer.
We are dividing the box $\Lambda$ into small cubes of size $\varepsilon$.
Let $t:= (c / \varepsilon)^\Dim$ be the total number of such cubes.
Therefore a configuration $\Conf$ in $\Lambda$ can be written as the collection 
$(\Conf_e)_{e \in E}$ where $E:=\Lambda\cap \varepsilon (\Z+ 1/2) ^\Dim $ and 
$\omega_e := \omega_{\Delta_e^\varepsilon}$, where 
$  \Delta_e^\varepsilon = e \oplus \varepsilon ]-1/2,1/2]^\Dim$.
For an enumeration $(e_1, \dots, e_t)$ of the cubes, we wrote $e_{[i]}=(e_1, \dots, e_i)$ and $\Conf_{e_{[i]}}=(\Conf_{e_1}, \dots, \Conf_{e_i})$.

Consider a Boolean function $f:\ConfSpace \to \{0,1\}$ and consider a decision tree $T$ determining the value $f(\Conf)$.
A decision tree queries the configuration $\Conf$ one cube after the other.
Hence to a random configuration is associated a random ordering of the cubes  $\egras=(\egras_1, \dots ,\egras_t)$.
It starts from a deterministic cube $\egras_1$ and looks at the configuration $\Conf_{\egras_1}$.
Then it chooses a second cube $\egras_2$ depending on $\egras_1$ and $\Conf_{\egras_1}$ and carries out.
At step $i>1$ the cubes $\egras_{[i-1]}$ have been visited and the configuration 
$\Conf_{\egras_{[i-1]}}$ is known.
The next cube $\egras_i$ to be explored is then expressed as a deterministic function of what have been already explored, i.e.
\begin{align}
\label{eq_function_prochain_cube}
\egras_i=\phi_i(\egras_{[i-1]},\Conf_{\egras_{[i-1]}}).
\end{align}
We then define the (random) time $\tempsarret$ that the algorithm takes to determine the function $f$, meaning
\begin{align}
\label{eq_temps_arret_definition}
\tempsarret (\Conf) = \min\{i\geq 1, \forall \Conf', 
\Conf'_{\egras_{[i]}} = \Conf_{\egras_{[i]}} \Rightarrow f(\Conf)=f(\Conf') \}.
\end{align}
\begin{theorem}[OSSS inequality]
\label{theo_OSSS}
If the function $0\leq f \leq 1$ is increasing, then for all configurations $\hat{\Conf}_{\Lambda^c}$,
\begin{align}
\label{eq_osss}
\text{Var}_{\Specification^{\Intensity, \beta}_{\Lambda,\hat{\Conf}_{\Lambda^c}}}(f)
\leq 2 \underset{e \in E}{\sum}
\delta(e,T)
\text{Cov}_{\Specification^{\Intensity, \beta}_{\Lambda,\hat{\Conf}_{\Lambda^c}}}
(f,\# \Conf_e)
+ \mathcal{O}(\varepsilon^{d}),
\end{align}
where 
$\delta(e,T):= 
\Specification^{\Intensity, \beta}_{\Lambda,\hat{\Conf}_{\Lambda^c}} 
(\exists i \leq \tempsarret, \egras_i=e) $ is called the \emph{revealment} of $e$ and is the probability that the cube $e$ is needed to determine the value of the function $f$.
\end{theorem}

\begin{remark}
The term
$\text{Cov}_{\Specification^{\Intensity, \beta}_{\Lambda,\hat{\Conf}_{\Lambda^c}}}
(f,\# \Conf_e)$
is non-negative because the specification
${\Specification^{\Intensity, \beta}_{\Lambda,\hat{\Conf}_{\Lambda^c}}}$, for all boundary conditions, satisfies the FKG inequality.
This is also a consequence of Theorem 1.1 in \cite{georgii_kuneth}
\end{remark}

\subsubsection{Proof of Theorem \ref{theo_OSSS}}
The original proof of the OSSS inequality, see \cite{O'Donnell_Saks_Schramm_Servedio_2005} or \cite{DuminilCopin_2017_Lecture_Ising_Potts} for more probabilistic version, uses the product structure of the space considered (the Bernoulli percolation model).
But $\Specification^{\Intensity, \beta}_{\Lambda,\hat{\Conf}_{\Lambda^c}}$ is not a product measure and we need a more elaborate method.
For this we will generalize the idea from 
\cite[Lemma 2.1]{DuminilCopin_Raoufi_Tassion_2019_Sharpness_perco_RCM} which sampled a finite family of dependant random variables, one after the other in a random order, using independent uniform variables.

In the continuum setting of point processes this simple idea is much harder to implement.
We will use ideas from the theory of stochastic domination.
Indeed the stochastic domination
$\Specification^{\Intensity, \beta}_{\Lambda,\hat{\Conf}_{\Lambda^c}} 
\preceq \Poisson^{\Intensity}_{\Lambda}$ from Proposition \ref{propo_dom_sto_poisson_area} implies, using Strassen's Theorem, that a configuration 
$\Conf \sim \Specification^{\Intensity, \beta}_{\Lambda,\hat{\Conf}_{\Lambda^c}}$ can be obtained from a dominating Poisson configuration  $\Conf^{D} \sim \Poisson_{\Lambda}^\Intensity$ by a random thinning of the dominating configuration, deciding for each $x \in \Conf^{D}$ if it belongs to the thinned configuration $\Conf$.
We will use an explicit form of the thinning probability proved in \cite{Hofer-temmel_Houdebert_2018}.

To formalize the thinning decision we are adding to each point of the configuration 
$\Conf^{D}$ an independent uniform mark between $0$ and $1$.
The marked configuration is denoted by $\Conf^{D,U}$ and its law is simply a marked Poisson point process.

Even if $\Conf^{D,U}$ is a marked configuration on $\Lambda$, the sampling procedure we will construct below needs as many dominating configurations $\Conf^{D,U}$ as there is cubes $e_i$ in $\Lambda$, meaning $t$.
Let us write $\Conf^{D,U,\otimes}= (\Conf^{D,U,1},\dots, \Conf^{D,U,t})$.
The law of $\Conf^{D,U,\otimes}$ is a product of marked Poisson point processes.

\begin{definition}
\label{defi_ordre_conf_marked}
For two marked configurations $\Conf^{D,U} , \tConf^{D,U}$, we say that 
$\Conf^{D,U}$ is smaller than $\tConf^{D,U}$, writing
$\Conf^{D,U} \leq \tConf^{D,U}$,
if for each $(x,u) \in \Conf^{D,U}$ we have
$(x,u) \in \tConf^{D,U}$.

We write $\Conf^{D,U,\otimes} \leq \tConf^{D,U, \otimes}$ if for all $i$ we have $\Conf^{D,U,i} \leq \tConf^{D,U, i}$
\end{definition}
{ 
\begin{proposition}
\label{propo_construction_area_thinning}
For each $(e_1,\dots, e_t)$, there exists a function $F_{(e_1\dots e_t)}$ indexed by an enumeration of the cubes, such that:
\begin{enumerate}
\item
if $\egras=(\egras_1, \dots ,\egras_t)$ is a random sequence such that for all $i$, 
 $\Conf^{D,U,i}$ is independent of $( \egras_{[i]}, \Conf^{D,U,1 .. i-1})$,
where $\Conf^{D,U,1 .. i-1}$ is a short-hand for the family 
$(\Conf^{D,U,1},\dots,\Conf^{D,U,i-1})$, then 
$$\Conf=F_{\egras}(\Conf^{D,U,\otimes}) \sim 
\Specification^{\Intensity, \beta}_{\Lambda,\hat{\Conf}_{\Lambda^c}},$$ 
where $\Conf^{D,U,\otimes}$ is the random element described just before Definition \ref{defi_ordre_conf_marked}.
\item
For fixed $(e_1\dots e_t)$, the function $F_{(e_1\dots e_t)}$  is increasing with respect to order defined in Definition \ref{defi_ordre_conf_marked}.
\end{enumerate}
\end{proposition}
\begin{proof}
Let us first construct the function $F_{(e_1\dots e_t)}$.

The configuration $\Conf$ is constructed in each cube $e_i$ one by one.
At each step $i$ the configuration $\Conf_{e_{[i-1]}}$ is already sampled and we are constructing $\Conf_{e_{i}}$.
The construction is taken from 
\cite[Proposition 4.1]{Hofer-temmel_Houdebert_2018} which gives explicitly the thinning probability for sampling a Gibbs point process dominated by a Poisson point process.

We are going to construct first a configuration (which we wrote $\Conf'$) on 
$$\tilde{\Lambda}_i := \Lambda \setminus 
(\Delta_{e_1}^\varepsilon \cup \dots \cup \Delta_{e_{i-1}}^\varepsilon)$$
according to the specification
$\Specification^{\Intensity,\beta}_{\tilde{\Lambda}_i,
\hat{\Conf}_{\Lambda^c}  \cup \Conf_{e_{[i-1]} }}$
and then only keep the points inside $\Delta_{e_i}^\varepsilon$ by setting 
$\Conf_{e_{i}}=\Conf'_{e_i}$.

To construct $\Conf'$ we consider the dominating configuration $\Conf^{D,U,i}_{\tilde{\Lambda}_i}$ restricted to the region where we are constructing the configuration.
Consider on $\tilde{\Lambda}_i$ a lexicographic order, which orders the points $(x,u) \in \Conf^{D,U,i}_{\tilde{\Lambda}_i}$.
The marks play no role in the ordering of the configuration.
The configuration $\Conf'$ is then constructed inductively the following way:
\begin{itemize}
\item at the beginning of the induction 
we set $\Conf'= \emptyset$.
\item
Then we consider each point $(x,u) \in \Conf^{D,U,i}_{\tilde{\Lambda}_i}$ one after the other with respect to the lexicographic order, and  if
\begin{align*}
p_{i} \left( 
x, \Conf' \cup \Conf_{e_{[i-1]}} \cup 
\hat{\Conf}_{\Lambda^c} \right)
\geq u
\end{align*}
we add 
$x$ to the configuration $\Conf'$, i.e. $\Conf'\gets \Conf' \cup \{x\}$.
The function $p_{i}$, whose expression comes from \cite[Proposition 4.1]{Hofer-temmel_Houdebert_2018}, is defined as
\begin{align}
\label{eq_thinning}
p_{i} \left( x,\Conf \right)
=
e^{
- \beta \Hamiltonian_{\{x\}}( x \cup \Conf )
} \times
\frac{\PartitionFunction (z,\beta ,\tilde{\Lambda}_i \cap ]x,\infty[,\Conf \cup x) }
{\PartitionFunction (z,\beta ,\tilde{\Lambda}_i \cap ]x,\infty[,\Conf )},
\end{align}
where the interval $]x,\infty[$ is defined with respect to the lexicographic ordering on $\tilde{\Lambda}_i$, and where $\Conf$ is a configuration such that 
$\Conf \cap \tilde{\Lambda}_i \cap ]x,\infty[=\emptyset$.
\end{itemize}
And finally at the end we set $\Conf_{e_{i}}=\Conf'_{e_i}$.

\begin{remark}
The Proposition 4.1 from \cite{Hofer-temmel_Houdebert_2018}
applies for general models, with random radii marks, for which the Hamiltonian is additive and satisfies a domination assumption.
The additivity of the Hamiltonian was already discussed in Remark \ref{remark_additivity_hamiltonian}, and we already used implicitely the domination condition to derive Proposition \ref{propo_dom_sto_poisson_area}.
Hence by considering constants radii, we can apply Proposition 4.1 from \cite{Hofer-temmel_Houdebert_2018} in our context.
\end{remark}


From the compatibility of the Gibbs specification \eqref{eq_compatibility_specification},
the assumption on $\egras$ and the Proposition 4.1 from \cite{Hofer-temmel_Houdebert_2018}, the sampled configuration has the law of $\Specification^{\Intensity,\beta}_{\Lambda,
\hat{\Conf}_{\Lambda^c}}$.

In order to prove that $F_{(e_1\dots e_t)}$  is increasing, It remains to prove that 
$p_i \left( x,\Conf \right)$ is increasing in $\Conf$.
Let us write $\tilde{\Lambda}_i^x:=
\tilde{\Lambda}_i \cap ]x,\infty[$.
 Using the additivity of the Hamiltonian, see Remark \ref{remark_additivity_hamiltonian},
the function $p_{i}$ can be rewritten as
\begin{align*}
p_{i} \left( x,\Conf \right)
&=
\frac{\int \Intensity^{\# \gamma} 
e^{-\beta \Hamiltonian_{\tilde{\Lambda}_i^x}(\gamma \cup \Conf) } 
e^{-\beta \Hamiltonian_{\{x\}}(x \cup \gamma \cup \Conf)}
\Poisson_{\tilde{\Lambda}_i^x} (d\gamma)}
{\PartitionFunction (z,\beta ,\tilde{\Lambda}_i^x,\Conf )}
\\ & =
\int  
e^{-\beta \Hamiltonian_{\{x\}}(x \cup \gamma \cup \Conf)}
\Specification_{\tilde{\Lambda}_i^x, \Conf}^{\Intensity, \beta} ( d \gamma ).
\end{align*}
Now using the fact that the integrated function is increasing in $\Conf$ with the stochastic domination from Proposition \ref{propo_dom_sto_poisson_area}, we have that the function $p_{i}$ is increasing in $\Conf$.
\end{proof}
}
\begin{remark}
The Proposition \ref{propo_construction_area_thinning} is the main improvement from the theory of randomized algorithm from Duminil-Copin, Raoufi and Tassion
\cite{DuminilCopin_Raoufi_Tassion_2019_Sharpness_perco_Voronoi,DuminilCopin_Raoufi_Tassion_2019_Sharpness_perco_RCM,
DuminilCopin_Raoufi_Tassion_2018_Sharpness_perco_Boolean_Model}.
Considering the assumption on $\egras$, the Proposition \ref{propo_construction_area_thinning} applies in particular when $\egras$ is independent of $\Conf^{D,U, \otimes}$ or when $\egras$ is constructed from $\Conf^{D,U, \otimes}$ as in \eqref{eq_function_prochain_cube}.

The proof of Proposition \ref{propo_construction_area_thinning} relies only 
on the additivity of the Hamiltonian, as well as  the fact that the function 
$\Hamiltonian_{\{ x \}}(x \cup \Conf)$, often called \emph{local energy}, is 
\begin{itemize}
\item uniformly bounded from below, which provides the stochastic domination of 
$\Specification^{\Intensity,\beta}_{\Lambda,{\Conf}_{\Lambda^c} }$ by the Poisson point process $\Poisson^{\Intensity}_{\Lambda}$ and allow us to apply Proposition 4.1 of 
\cite{Hofer-temmel_Houdebert_2018}.

\item decreasing with respect to $\Conf$, which provides the monotonicity of $F_{\egras}$ thanks to \cite[Proposition 4.1]{Hofer-temmel_Houdebert_2018}.

\end{itemize} 
So Proposition \ref{propo_construction_area_thinning}, and more generally Theorem \ref{theo_sharpness_perco_mesure_wired} would trivially generalised to every Gibbs measure whose interaction satisfies those two properties.
While the first property is a standard property satisfied by most interactions considered in the literature,
the second property, related to the \emph{monotony} of the Gibbs specification, is less common.
To the best of our knowledge the area-interaction is the only interaction considered in the literature which satisfies this property.
\end{remark}

Now consider two independent configurations $\Conf^{D,U,\otimes}$ and $\tConf^{D,U,\otimes}$.
The random ordering of cubes $\egras=(\egras_1,\dots, \egras_t)$ considered starting now is constructed from $\Conf^{D,U,\otimes}$ with \eqref{eq_function_prochain_cube}.
We write $\Conf = F_\egras(\Conf^{D,U,\otimes})$ and 
$\tConf = F_\egras(\tConf^{D,U,\otimes})$.
Thanks to Proposition \ref{propo_construction_area_thinning}, those are two realisations of 
$\Specification^{\Intensity, \beta}_{\Lambda,\hat{\Conf}_{\Lambda^c}}$.
 Even if they are both sample from $\egras$ (which depends on $\Conf$), the configurations $\Conf$ and $\tConf$ are independent since $\egras$ intervenes only in the order of the cubes $\Delta_{e_i}$ and the configurations in $\Delta_{e_i}$ are sampled from independent Poisson configurations.

Now write for $i \leq \tempsarret = \tempsarret ( \Conf)$
\begin{align*}
\gamma^i=F_{\egras}(
\tConf^{D,U,1},\dots , \tConf^{D,U,i},\Conf^{D,U,i+1},\dots ,
\Conf^{D,U,\tempsarret}, \tConf^{D,U,\tempsarret +1}, \dots , \tConf^{D,U,t}
).
\end{align*}
Then since $0\leq f \leq 1$ we have
\begin{align*}
\text{Var}_{\Specification^{\Intensity, \beta}_{\Lambda,\hat{\Conf}_{\Lambda^c}}}[f]
\leq
\E_{\Specification^{\Intensity, \beta}_{\Lambda,\hat{\Conf}_{\Lambda^c}}}
[\ |f - 
\E_{\Specification^{\Intensity, \beta}_{\Lambda,\hat{\Conf}_{\Lambda^c}}}[f] |\ ]
=
\E [\ | f(\gamma^0) - \E[f(\gamma^\tempsarret)] |\ ],
\end{align*}
where in the right-hand side the expectation is with respect to the two independent marked Poisson realisations $\Conf^{D,U,\otimes}$ and $\tConf^{D,U,\otimes}$ from which the $\gamma^i$ are constructed.
Then 
\begin{align*}
\text{Var}_{\Specification^{\Intensity, \beta}_{\Lambda,\hat{\Conf}_{\Lambda^c}}}[f]
& \leq 
\E [\ | \E \left[f(\gamma^0)|\Conf^{D,U,\otimes}\right] - 
\E\left[f(\gamma^\tempsarret) | \Conf^{D,U,\otimes}\right] |\ ]
\\ & \leq
\E [\ |  f(\gamma^0)  - f(\gamma^\tempsarret)  |\ ]
\\ & \leq
\sum_{i=1 \dots t}
\E [\ |  f(\gamma^i)  - f(\gamma^{i-1})   |\ \1_{i \leq \tempsarret } ]
\\ & =
\sum_{i=1 \dots t}   \sum_{e \in E}
\E [\E [ |  f(\gamma^i)  - f(\gamma^{i-1})| \ |\ 
\Conf^{D,U,1 .. i-1}  ]
\ \1_{i \leq \tempsarret } \ \1_{\egras_i = e} ],
\end{align*}
where $\Conf^{D,U,1 .. i-1}$ is, as defined before, a short-hand for 
$(\Conf^{D,U,1},\dots,\Conf^{D,U,i-1})$.
\begin{lemma}
\label{lemme_inegalité_covariance}
On the event $\{i \leq \tempsarret \} \cap \{ \egras_i = e \}$ we have
\begin{align*}
\E [ |  f(\gamma^i)  - f(\gamma^{i-1})| \ |\ 
\Conf^{D,U,1 .. i-1} ]
\leq
2 \text{Cov}_{\Specification^{\Intensity, \beta}_{\Lambda,\hat{\Conf}_{\Lambda^c}}}
(f,\# \Conf_e)
+ \mathcal{O}(\varepsilon^{2d}).
\end{align*}
\end{lemma}
Using the Lemma \ref{lemme_inegalité_covariance} we get
\begin{align*}
\text{Var}_{\Specification^{\Intensity, \beta}_{\Lambda,\hat{\Conf}_{\Lambda^c}}}[f]
& \leq 
\sum_{i=1 \dots t}   \sum_{e \in E}
(\
2 \text{Cov}_{\Specification^{\Intensity, \beta}_{\Lambda,\hat{\Conf}_{\Lambda^c}}}
(f,\# \Conf_e)
+ \mathcal{O}(\varepsilon^{2d}) 
\ )
\E [\ \1_{i \leq \tempsarret } \ \1_{\egras_i = e} ]
\\ & =
\sum_{e \in E}
(\
2 \text{Cov}_{\Specification^{\Intensity, \beta}_{\Lambda,\hat{\Conf}_{\Lambda^c}}}
(f,\# \Conf_e)
+ \mathcal{O}(\varepsilon^{2d}) 
\ )
\delta(e,T)
\\ & =
2 \left( \sum_{e \in E}
 \text{Cov}_{\Specification^{\Intensity, \beta}_{\Lambda,\hat{\Conf}_{\Lambda^c}}}
 (f,\# \Conf_e)
\delta(e,T)
\right)
+ \mathcal{O}(\varepsilon^{d}),
\end{align*}
where the last equality uses that the cardinal of $E$ is of order $1/ \varepsilon^\Dim$.
\begin{proof}[Proof of Lemma \ref{lemme_inegalité_covariance}]
Since $f$ is a positive increasing function bounded by one, we have
\begin{align*}
|f(\gamma^i)  - f(\gamma^{i-1})|
&= 
| f(\gamma^i)  - f(\gamma^{i-1}) |
\times
\left( 1 - \1_{\#\gamma^i_e = \# \gamma^{i-1}_e =0 }  \right)
\\ &  \leq
 \left( f(\gamma^i)  - f(\gamma^{i-1}) \right ) \times
\left( \1_{\# \gamma^i_e \geq 1} - \1_{\# \gamma^{i-1}_e \geq 1}  \right)
+
\1_{\# (\gamma^{i}_e \cup \gamma^{i-1}_e) > 1},
\end{align*}
where we used the fact that if 
$\#\gamma^{i-1}_e= \#\gamma^{i}_e=0$, then $ \gamma^{i-1}=\gamma^{i}$.

But using the stochastic domination of Proposition \ref{propo_dom_sto_poisson_area}  we have the following easy bound
\begin{align*}
\E \left[ \1_{\# (\gamma^{i}_e \cup \gamma^{i-1}_e) > 1}
\ |\ \Conf^{D,U,1 .. i-1} \right]
\leq \Poisson^{2\Intensity}(\# \Conf_e >1)
=\mathcal{O}(\varepsilon^{2d}) .
\end{align*}
{ Using the first item of Proposition \ref{propo_construction_area_thinning} we obtain that 
$\gamma^i,\gamma^{i-1} \sim 
\Specification^{\Intensity, \beta}_{\Lambda,\hat{\Conf}_{\Lambda^c}}$ for all 
$ \Conf^{D,U,1 .. i-1} $ (as those variables don't intervene in the definition of 
$\gamma^{i-1}, \gamma^{i}$ but only in $\egras_1, \dots , \egras_i$) and therefore}
\begin{align*}
\E \left[  f(\gamma^{i-1})
\1_{\# \gamma^{i-1}_e \geq 1}
\ |\ \Conf^{D,U,1 .. i-1} \right]
& =
\E \left[  f(\gamma^{i})
\1_{\# \gamma^{i}_e \geq 1}
\ |\ \Conf^{D,U,1 .. i-1} \right]
\\ & =
\E_{\Specification^{\Intensity, \beta}_{\Lambda,\hat{\Conf}_{\Lambda^c}}}
\left[ f(\Conf) \1_{\Conf_e \geq 1} \right].
\end{align*}
{
Now let study the term
$\E \left[  f(\gamma^{i-1}) \1_{\# \gamma^{i}_e \geq 1}
\ |\ \Conf^{D,U,1 .. i-1} \right]$.
One can no longer apply directly the monotonicity of Proposition \ref{propo_construction_area_thinning} since changing $\Conf^{D,U,\otimes}$ 
changes also $\egras$.
But, as in \cite{DuminilCopin_Raoufi_Tassion_2019_Sharpness_perco_RCM}, by conditioning by $\Conf^{D,U,\otimes}$ and applying the FKG inequality
(applied to $\tConf^{D,U,\otimes}$)  we obtain}
\begin{align*}
\E &\left[  f(\gamma^{i-1})
\1_{\# \gamma^{i}_e \geq 1}
\ |\ \Conf^{D,U,1 .. i-1} \right]
\\ & \hspace{1cm}  =
\E \left[  \E \left[
f(\gamma^{i-1})
\1_{\# \gamma^{i}_e \geq 1}
\ |\ \Conf^{D,U,\otimes} \right] 
|\ \Conf^{D,U,1 .. i-1} \right]
\\ & \hspace{1cm}  \geq 
\E \left[  \E \left[
f(\gamma^{i-1})
\ |\ \Conf^{D,U,\otimes} \right] 
  \E \left[
\1_{\# \gamma^{i}_e \geq 1}
\ |\ \Conf^{D,U,\otimes} \right] 
|\ \Conf^{D,U,1 .. i-1} \right]
\\ & \hspace{1cm}  =
\E_{\Specification^{\Intensity, \beta}_{\Lambda,\hat{\Conf}_{\Lambda^c}}}
\left[ f(\Conf)  \right]
\E_{\Specification^{\Intensity, \beta}_{\Lambda,\hat{\Conf}_{\Lambda^c}}}
\left[ \1_{\Conf_e \geq 1} \right],
\end{align*}
where the last equality uses the  measurability  of
$ \E \left[
\1_{\# \gamma^{i}_e \geq 1}
\ |\ \Conf^{D,U,\otimes} \right] $
with respect to $\Conf^{D,U,1 .. i-1}$, and the first item of Proposition \ref{propo_construction_area_thinning}.

The same is true for
\begin{align*}
\E \left[  f(\gamma^{i})
\1_{\# \gamma^{i-1}_e \geq 1}
\ |\ \Conf^{D,U, 1 .. i-1} \right]
& =
\E \left[
\E \left[  f(\gamma^{i})
\1_{\# \gamma^{i-1}_e \geq 1}
\ |\ \Conf^{D,U, 1 .. i} \right]
| \ \Conf^{D,U, 1 .. i-1}
\right]
\\ & \geq 
\E_{\Specification^{\Intensity, \beta}_{\Lambda,\hat{\Conf}_{\Lambda^c}}}
\left[ f(\Conf)  \right]
\E_{\Specification^{\Intensity, \beta}_{\Lambda,\hat{\Conf}_{\Lambda^c}}}
\left[ \1_{\Conf_e \geq 1} \right]
,
\end{align*}
and therefore
\begin{align*}
\E [ |  f(\gamma^i)  - f(\gamma^{i-1})| \ |\ 
\Conf^{D,U,1 .. i-1} ]
& \leq 2 
\text{Cov}_{\Specification^{\Intensity, \beta}_{\Lambda,\hat{\Conf}_{\Lambda^c}}}
(f,\1_{\Conf_e \geq 1})
+ \mathcal{O}(\varepsilon^{2d})
\\ & \leq 2 
\text{Cov}_{\Specification^{\Intensity, \beta}_{\Lambda,\hat{\Conf}_{\Lambda^c}}}
(f,\# \Conf_e)
+ \mathcal{O}(\varepsilon^{2d}),
\end{align*}
where the last inequality coming from the use of the FKG inequality, valid for 
$\Specification^{\Intensity,\beta}_{\Lambda,\hat{\Conf}_{\Lambda}^c}$ as a consequence of Theorem 2.1 in \cite{georgii_kuneth}, 
and applied to $f$ and $\#\Conf_{e} - \1_{\Conf_e \geq 1}$.

The result is proved.
\end{proof}
\subsection{Proof of Theorem \ref{theo_sharpness_perco_mesure_wired}}
\label{section_utilisation_OSSS}
We are now going to prove
Theorem \ref{theo_sharpness_perco_mesure_wired} by applying Theorem \ref{theo_OSSS}.
We need the following classical lemma in the theory of randomised algorithms (see Lemma 3 in \cite{DuminilCopin_Raoufi_Tassion_2019_Sharpness_perco_Voronoi} for instance).
\begin{lemma}
\label{lemme_preuve_sharpness_perco}
Consider a converging  sequence of increasing differentiable functions 
$g_n: ]0,\Intensity_{max}] \to ]0,1]$  satisfying for all $n\geq 1$
\begin{align}
\label{eq_lemma_preuve_shaprness_perco}
g_n'(\Intensity) 
\geq \alpha_2 \frac{n}{\Sigma_n (\Intensity)} g_n(\Intensity),
\end{align}
where $\alpha_2>0$ is a positive constant and 
\begin{align*}
\Sigma_n (\Intensity)
=
\sum_{i=0}^{n-1} g_i(\Intensity).
\end{align*}
Then there exists $\tilde{\Intensity} \in [0,\Intensity_{max}]$ such that:
\begin{itemize}
\item
For every $ \Intensity < \tilde{\Intensity}$, there exists $\alpha_1:=\alpha_1(\Intensity)>0$ such that for all $n$,
\begin{align*}
g_n(\Intensity) \leq \exp(-\alpha_1 n).
\end{align*}
\item
For every $ \Intensity > \tilde{\Intensity}$,
\begin{align*}
g(\Intensity)
:= \lim_{n \to \infty} g_n(\Intensity)
\geq 
\alpha_2  \left( \Intensity - \tilde{\Intensity} \right).
\end{align*}
\end{itemize}
\end{lemma}
{\begin{remark}
The statement of Lemma \ref{lemme_preuve_sharpness_perco} is slightly different from the original one found in \cite{DuminilCopin_Raoufi_Tassion_2019_Sharpness_perco_RCM} as we added a constant $\alpha_2$ in \eqref{eq_lemma_preuve_shaprness_perco} in order to apply Lemma \ref{lemme_preuve_sharpness_perco} directly to $g_n=\theta_n$.

Furthermore in the original statement of
\cite{DuminilCopin_Raoufi_Tassion_2019_Sharpness_perco_RCM},
the first point is obtained for $n$ large enough.
However since the functions we consider in Lemma \ref{lemme_preuve_sharpness_perco} are bounded by one, the same proof as in 
\cite{DuminilCopin_Raoufi_Tassion_2019_Sharpness_perco_RCM} proves the result for all $n$.

\end{remark}
}
Therefore in order to prove Theorem \ref{theo_sharpness_perco_mesure_wired} it is sufficient to prove that the functions $g_n=\theta_n$ satisfies the assumptions  of Lemma \ref{lemme_preuve_sharpness_perco}.
By construction, the functions $\Intensity \mapsto \theta_n(\Intensity)$ are increasing.
The following lemma proves the differentiability of the functions $\theta_n$.
\begin{lemma}
\label{lemme_derivation_gibbs}
For all functions $f$ measurable and bounded,  and all $\Intensity >0$,
\begin{align*}
\frac{d}{d\Intensity} P^{\Intensity, \beta}_{n,wired} (f)
= 
\frac{1}{\Intensity}
\text{Cov}_{P^{\Intensity, \beta}_{n,wired}}(f,\# \Conf).
\end{align*}
\end{lemma}
The proof of this result is done in the annex Section \ref{section_annex}.
Therefore the only remaining task in order to prove Theorem \ref{theo_sharpness_perco_mesure_wired} is to prove that the functions $\theta_n$ satisfy \eqref{eq_lemma_preuve_shaprness_perco}.
This is done using the OSSS inequality \eqref{eq_osss} to the wired measure 
$\mu_n^\Intensity$ for the Boolean function 
$f(\Conf)=\1_{0 \Connected{r}\partial \Lambda_n}(\Conf)$
to the well chosen algorithms from \cite{DuminilCopin_Raoufi_Tassion_2019_Sharpness_perco_Voronoi,
DuminilCopin_Raoufi_Tassion_2018_Sharpness_perco_Boolean_Model}.
\begin{proposition}
\label{propo_revealment_borne}
For $0 \leq s \leq n$ {and $\epsilon$ small}, there exists an algorithm 
$T_s$ { adapted to the event considered } such that 
we have $\delta(e,T)=0$ if 
$e \not \in \Lambda_{n+r+1}$ and otherwise
\begin{align}
\label{eq_revealment_borne}
\delta(e,T_s)
\leq
\mu_n^\Intensity \left(
\Delta_e^\varepsilon \oplus B_r(0) \Connected{r} \partial \Lambda_s 
\right).
\end{align}
\end{proposition}
This proposition uses the now standard algorithms used in \cite{DuminilCopin_2017_Lecture_Ising_Potts,
DuminilCopin_Raoufi_Tassion_2019_Sharpness_perco_Voronoi,
DuminilCopin_Raoufi_Tassion_2019_Sharpness_perco_RCM}.
The proof of Proposition \ref{propo_revealment_borne} is done in the Annex Section \ref{section_annex}.

Using Theorem \ref{theo_OSSS}, summing over $s$ between $0$ and $n-1$ and divinding by $n$  we get
\begin{align*}
\theta_n(z) (1-\theta_n(z) )
& \leq
\frac{2}{n} \sum_{e \in E}
\sum_{s=0}^{n-1} 
\mu_n^\Intensity \left(
\Delta_{e}^\varepsilon \oplus B_r(0) \Connected{r} \partial \Lambda_s 
\right)
\text{Cov}_{\mu_n^\Intensity}(f,\# \Conf_e)
+ \mathcal{O}(\varepsilon^{d})
\\ & \leq
\frac{2}{n} \sum_{e \in E}
\sum_{s=0}^{n-1} 
\mu_n^\Intensity \left(
B_{r+1}(e) \Connected{r} \partial \Lambda_s 
\right)
\text{Cov}_{\mu_{n}^\Intensity}(f,\# \Conf_e)
+ \mathcal{O}(\varepsilon^{d}),
\end{align*}
where the last inequality is valid by considering $\varepsilon$ small enough.
To see if $B_{r+1}(e)$ is $r$-connected to $\partial \Lambda_s$, it is enough to check if at least one point $y$, belonging to a finite collection of points $Y_r^e$, is $r$-connected to $\partial \Lambda_s$. 
Actually the set $Y_r^e$ contains points close to the boundary of $B_{r+1}(e)$ (inside and outside) and also a point in $B_{r+1}(e)\cap\partial \Lambda_s$ if the intersection is non empty. It is easy to see also that the set $Y_r^e$ can be chosen with cardinal $\alpha_r$ depending only on $r$ and $d$.

Therefore
\begin{align}
\label{eq_preuve_sharpness_majoration_variance}
\theta_n(z) (1-\theta_n(z) )
& \leq
\frac{2}{n} \sum_{e \in E}
\sum_{y \in Y_r^e}\sum_{s=0}^{n-1} 
\mu_n^\Intensity \left(
y \Connected{r} \partial \Lambda_s 
\right)
\text{Cov}_{\mu_{n}^\Intensity}(f,\# \Conf_e)
+ \mathcal{O}(\varepsilon^{d}).
\end{align}
But we have
\begin{align*}
\sum_{s=0}^{n-1}  \mu_n^\Intensity 
\left(y  \Connected{r} \partial \Lambda_s \right) 
& \leq
\sum_{s=0}^{n-1} \mu_n^\Intensity 
\left( y \Connected{r} \partial 
\Lambda_{\lfloor   s-||y||  \rfloor }(y) \right)
\\ & \leq
2 \sum_{s=0}^{n-1} \mu_n^\Intensity 
\left( y  \Connected{r} \partial \Lambda_{s}(y) \right)
\\ & \leq
4 \sum_{s=0}^{n/2}  \mu_n^\Intensity 
\left( y  \Connected{r} \partial \Lambda_{s}(y) \right),
\end{align*}
where $\lfloor  s  \rfloor$ stands for the floor function of the absolute value of $s$, and
where $\Lambda_s(y)=y \oplus \Lambda_s$ is the cube translated by the vector $y$.
Now using the stochastic domination between wired measures 
$\mu_n^\Intensity \preceq \mu_s^\Intensity$
from a straightforward extension of Proposition \ref{propo_dom_sto_free_wired_boite}, we finally obtain
\begin{align}
\label{eq_preuve_sharpness_majoration_revealment}
\sum_{s=0}^{n-1}  \mu_n^\Intensity 
\left(y   \Connected{r} \partial \Lambda_s(y) \right) 
& \leq
4 \sum_{s=0}^{n/2}   \mu_s^\Intensity 
\left(0  \Connected{r} \partial \Lambda_{s} \right)
\leq 
4 \Sigma_{n}(\Intensity).
\end{align}
\begin{remark}
Equation \eqref{eq_preuve_sharpness_majoration_revealment} is the only step of the proof of Theorem \ref{theo_sharpness_perco_mesure_wired} where we used the fact that $\mu_n^\Intensity$ is a wired area-interaction measure.

The purpose of summing up to $n/2$ is to ensure that 
$\Lambda_{3s+2}( y ) \subseteq \Lambda_{3n+2}$ and therefore
$\mu_n^\Intensity \preceq \mu_s^\Intensity$.
This is why we considered $\Lambda_{3n+2}$ in the definition on $\mu_n^\Intensity$.
\end{remark}

Hence from \eqref{eq_preuve_sharpness_majoration_variance},  \eqref{eq_preuve_sharpness_majoration_revealment} and Lemma \ref{lemme_derivation_gibbs}, we obtain
\begin{align*}
\theta_n(\Intensity) (1-\theta_n(\Intensity) )
&\leq
8 \alpha_r \frac{\Sigma_{n}(\Intensity)}{n} \sum_{e \in E}
\text{Cov}_{\mu_{n}^\Intensity}(f,\# \Conf_e)
+ \mathcal{O}(\varepsilon^d)
\\ & \leq
8 \alpha_r \frac{\Sigma_{n}(\Intensity)}{n} \text{Cov}_{\mu_{n}^\Intensity}(f,\# \Conf)
+ \mathcal{O}(\varepsilon^d)
\\ & \leq
8 \alpha_r \frac{\Sigma_{n}(\Intensity)}{n} \Intensity \theta_n'(\Intensity)
+ \mathcal{O}(\varepsilon^d).
\end{align*}
Now consider $\Intensity_{max}>  \IntensityThresholdAreaPercolation$.
Then for $\Intensity \le \Intensity_{max}$ we have that $1-\theta_n(\Intensity) \geq c>0$, where $c$ can be chosen uniformly in $n$.
Therefore by letting $\varepsilon$ goes to 0 we obtain

$$ \theta_n'(\Intensity) 
\geq \frac{c}{8 \Intensity_{max} \alpha_r} \frac{n}{\Sigma_n (\Intensity)} \theta_n(\Intensity)
:=
\alpha_2 \ \frac{n}{\Sigma_n (\Intensity)} \theta_n(\Intensity)$$

and equation \eqref{eq_lemma_preuve_shaprness_perco} is fulfilled. From Lemma \ref{lemme_preuve_sharpness_perco} we get the existence of a threshold $\tilde{\Intensity}$ but from the conclusion of Lemma \ref{lemme_preuve_sharpness_perco} this threshold has to be the percolation threshold $\IntensityThresholdAreaPercolation$.
The proof of Theorem \ref{theo_sharpness_perco_mesure_wired} is complete.
\subsection{Proof of proposition \ref{proposition_property_threshold}}
\label{section_preuve_zcbeta_ligget}
We are in this section proving proposition \ref{proposition_property_threshold}, meaning the regularity of the function $\beta\to \IntensityThresholdAreaPercolation(\beta,r)$ and  that for all $r>0$ and $\beta$ large enough (depending on $r$) we have 
$\IntensityThresholdAreaPercolation(\beta,r)=\beta$.
\subsubsection{Explicit value for the threshold}
To prove that 
$\IntensityThresholdAreaPercolation(\beta,r)=\beta$ for $\beta$ large enough, we are going to use the well-known \emph{Fortuin-Kasteleyn} representation of bicolor Widom-Rowlinson measures (and therefore area-interaction measures as well) by \emph{Continuum Random Cluster measures}, defined as follows.
\begin{definition}
A stationary measure $P^{crc}$ on $\ConfSpace$ is a Continuum Random Cluster measure of activity $\Intensity$, if for all bounded $\Lambda \subseteq \R^\Dim$ and all bounded measurable function $f$,
\begin{align}
\label{eq_DLR_CRCM}
\int f \ dP^{crc}
=
\int \int f(\Conf'_{\Lambda} \cup \Conf_{\Lambda^c})
\frac{2^{\NccLambda(\Conf'_{\Lambda} \cup \Conf_{\Lambda^c})}}
{\PartitionFunctionCRCM(\Intensity,\Lambda,\Conf_{\Lambda^c})}
\Poisson_{\Lambda}^{\Intensity}(d \Conf'_{\Lambda})
P^{crc}(d \Conf),
\end{align}
where 
$\NccLambda(\Conf)
= \lim_{\Delta \to \R^\Dim} 
\Ncc (\Conf_{\Delta})-\Ncc (\Conf_{\Delta \setminus \Lambda})$ 
with 
$\Ncc (\Conf_{\Delta})$ counting the number of connected components of
$B_{1/2} (\Conf_{\Delta})$, and
$\PartitionFunctionCRCM(\Intensity,\Lambda,\Conf_{\Lambda^c})$ being the standard non-degenerate partition function.
\end{definition}
{Let us emphasis on the fact that $\NccLambda(\Conf)$ can be negative, which is a major difficulty on the study of the Continuum Random Cluster Model.}
The existence of Random  Cluster measures for every activity $\Intensity$ was proved in \cite{dereudre_houdebert}.
In \cite{Houdebert_2017_percolation_CRCM} the author proves that for any $r>0$ there exists $\IntensityThresholdAreaNonUnicitySym^r$ such that for any 
$\Intensity>\IntensityThresholdAreaNonUnicitySym^r$ every Continuum Random Cluster measures $r$-percolates, which as a consequence gives the non uniqueness of area-interaction measures in the symmetric case for 
$\Intensity=\beta>\IntensityThresholdAreaNonUnicitySym^{1/2}$. 
Indeed the standard \emph{Fortuin-Kasteleyn} representation claims that keeping each finite $1/2$-connected component from a  Continuum Random Cluster measure with probability $1/2$ and   keeping as you want the infinite $1/2$-connected component, this construction produces an area-interaction measure with parameters $\Intensity$, $\beta=\Intensity$ (see \cite{chayes_kotecky,Houdebert_2017_percolation_CRCM}). 
Therefore, as recalled in Proposition \ref{propo_phase_transition_homogeneous_fk}, one area-interaction measure $1/2$-percolates and another one does not $1/2$-percolate. Consequently  $\IntensityThresholdAreaPercolation(\beta,1/2)=\beta$ for $\beta>\IntensityThresholdAreaNonUnicitySym^{1/2}$ and the result is proved for $r=1/2$.

For $r<1/2$, we notice that for $\beta>\IntensityThresholdAreaNonUnicitySym^{r}$ every Continuum Random Cluster measure, with activity $\beta$, $r$-percolates and therefore at least one area-interaction measure with parameters $z=\beta$ and $\beta$, $r$-percolates. That implies that $\IntensityThresholdAreaPercolation(\beta,r)\le \beta$ but by monotonicity $\IntensityThresholdAreaPercolation(\beta,r) \ge \IntensityThresholdAreaPercolation(\beta,1/2)=\beta$ which proves the result for $r<1/2$.

It remains the case  $r>1/2$ which is more delicate. From now let $\Intensity>\IntensityThresholdAreaNonUnicitySym^{1/2}$ and let $P^{crc}$ be a  Continuum Random Cluster measure of activity $\Intensity$, which is therefore $1/2$-percolating by the choice of $\Intensity$.
Consider the measure $P^{thin}$ obtained from $P^{crc}$ by removing all points $x\in\Conf$ belonging to the (unique) infinite component  of $B_{1/2}(\Conf)$.
Then as a consequence of the Fortuin-Kasteleyn representation,  we have the following domination:
\begin{align}
\label{eq_crcm_thinned_area}
P^{\Intensity, \Intensity}_{free} 
\preceq
P^{thin}.
\end{align}
By construction $P^{thin}$ does not $1/2$-percolates, since the $1/2$-infinite connected component was removed.
We are proving in the following lemma that by considering $z$ large enough, this removed infinite connected component prevents $r$-percolation in $P^{thin}$.
\begin{lemma}
\label{lemme_threshold_Un_Perco}
There exists 
$\IntensityThresholdAreaUnPerco\geq\IntensityThresholdAreaNonUnicitySym^{1/2}$
such that for all $\Intensity > \IntensityThresholdAreaUnPerco$, 
$P^{thin}$ does not $r$-percolate.
\end{lemma}
This lemma implies that 
$\IntensityThresholdAreaPercolation(\beta,r)\geq \beta$ for 
$\beta>\IntensityThresholdAreaUnPerco$ and by monotonicity
$\IntensityThresholdAreaPercolation(\beta,r)\leq \IntensityThresholdAreaPercolation(\beta,1/2) = \beta$. That concludes the proof.

\begin{proof}[Proof of Lemma \ref{lemme_threshold_Un_Perco}]
We are divinding the space $\R^\Dim$ into squares 
$C_y:=y\oplus]-r,r]^\Dim$, with $y\in 2r\Z^\Dim$.
We are going to prove that a lot of cubes are filled by balls from $B_{1/2}(\Conf)$, where $\Conf \sim P^{crc}$.
By applying a well-known result from Liggett,Schonmann and Stacey \cite{liggett_schonmann_stacey}, the infinite connected component of $B_{1/2}(\Conf)$ will be (for large activities) very thick.
This will prevent $P^{thin}$ to $r$-percolate.

By stationarity of $P^{crc}$ we are only considering the conditional probability
\begin{align*}
p_\Intensity(\Conf_{\Lambda_0^c})
=
P^{crc}(C_0 \subseteq B_{1/2}(\Conf) | \Conf_{\Lambda_0^c}),
\end{align*}
where $\Lambda_0:=]-r-1,r+1]^\Dim$.
For $C_0$ to be $1/2$-covered, it is sufficent that $\Conf$ has enough nicely placed points.
Therefore we are once again diving $C_0$ into smaller cubes $\tilde{C_i}$, $1\leq i \leq k$, of side smaller than 
$\frac{1}{2\sqrt\Dim}$.
If all those small cubes contains a point then $C_0$ will be $1/2$-covered.
Using the union bound we have
\begin{align*}
1-p_\Intensity(\Conf_{\Lambda_0^c})
\leq \sum_{i=1}^k
P^{crc}(\#\Conf_{\tilde{C_i}} = 0 | \Conf_{\Lambda_0^c}).
\end{align*}
But
\begin{align*}
P^{crc}(&\#\Conf_{\tilde{C_i}} =0 | \Conf_{\Lambda_0^c})
 \\ &=
\int \int
\1_{\#\Conf''_{\tilde{C_i}} =0}
\frac{2^{\Ncc^{\tilde{C}_i}(\Conf''_{\tilde{C_i}}
\cup\Conf'_{\Lambda_0\setminus \tilde{C}_i} \cup \Conf_{\Lambda_0^c})}}
{\PartitionFunctionCRCM \left(\Intensity, \tilde{C_i}, 
\Conf'_{\Lambda_0\setminus \tilde{C}_i} 
\cup \Conf_{\Lambda_0^c})\right)}
\Poisson_{\tilde{C}_i}^{\Intensity}(d\Conf''_{\tilde{C}_i})
P^{crc}(d \Conf'_{\Lambda_0} | \Conf_{\Lambda_0^c})
\\ &=
 \int
\frac{e^{-\Intensity \Leb (\tilde{C}_i) } }
{\PartitionFunctionCRCM \left(\Intensity, \tilde{C_i}, 
\Conf'_{\Lambda_0\setminus \tilde{C}_i} 
\cup \Conf_{\Lambda_0^c})\right)}
P^{crc}(d \Conf'_{\Lambda_0} | \Conf_{\Lambda_0^c}),
\end{align*}
where we used \eqref{eq_DLR_CRCM}, the compatibility of the Continuum Random Cluster specification and the fact that 
$\Ncc^{ \tilde{C}_i }\left( 
\emptyset  \cup
\Conf'_{\Lambda_0\setminus \tilde{C}_i} \cup \Conf_{\Lambda_0^c} \right)=0$.
But we also have the following bound on $\Ncc^{ \tilde{C}_i }(\Conf)$:
\begin{align*}
\# \Conf_{\tilde{C}_i} (1- c_d)
\leq
\Ncc^{ \tilde{C}_i }(\Conf)
\leq\# \Conf_{\tilde{C}_i},
\end{align*}
where $c_d$ is the kissing number in dimension $\Dim$, which is always larger than $2$.
This implies
\begin{align*}
P^{crc}(\#\Conf_{\tilde{C_i}} =0 | \Conf_{\Lambda_0^c})
\leq
e^{-\Intensity \Leb (\tilde{C}_i) } e^{\Intensity  \Leb (\tilde{C}_i)(1-2^{1-c_d}) }
=
e^{-2^{1-c_d}\Intensity  \Leb (\tilde{C}_i) },
\end{align*}
and therefore
\begin{align}
\label{eq_borne_liggett}
p_\Intensity(\Conf_{\Lambda_0^c}) 
\geq
1- k e^{-2^{1-c_d}\Intensity  \Leb (\tilde{C}_i) }.
\end{align}
The bound \eqref{eq_borne_liggett} is uniform in $\Conf_{\Lambda_0^c}$ and goes to $1$ as $\Intensity$ goes to infinity.
Therefore by applying the result of Liggett Schonmann and Stacey \cite{liggett_schonmann_stacey}, we have for $\Intensity$ large enough that the set of completely covered (by balls of radii $1/2$) boxes $C_y$ stochastically dominates an independent and identically  Bernoulli field with parameter $0<p<1$ as large as we want. 
Let $\mathcal{C}^{infinite}$ be the set of sites $y\in 2r \Z^\Dim$ belonging to the infinite connected component.

Now consider the set  
$\mathcal{C}^{finite}= 2r\Z^\Dim \setminus \mathcal{C}^{infinite} $ 
of sites $y\in 2r \Z^\Dim$ not belonging to the infinite cluster of this site percolation.
By considering $\Intensity$ large enough we have that $\mathcal{C}^{finite}$ only contains bounded connected components, with respect to the site percolation.

Now consider the configuration $\Conf^{finite}$ obtained from $\Conf$ by removing the points $x$ in the infinite connected component of $B_{1/2}(\Conf)$.
By construction we have that $B_{1/2}(\Conf^{finite})$ is entirely contained in the cubes $C_y$ for $y \in \mathcal{C}^{finite}$ and the configuration $\Conf^{finite}$ cannot $r$-percolates, since otherwise the infinite connected component would have to cross cubes $y \in \mathcal{C}^{infinite}$, which is not possible.

This implies that for $\Intensity$ large enough the measure $P^{thin}$ does not $r$-percolate.
\end{proof}

\subsubsection{Regularity of the threshold}

In this section we show that the function  $\beta\to \IntensityThresholdAreaPercolation(\beta,r)$ is non-decreasing and Lipschitz. From the previous section there exists $M>1$ such that $\IntensityThresholdAreaPercolation(\beta,r)=\beta$ for $\beta\ge M$. Therefore it is sufficient to show that the function $\beta\mapsto \IntensityThresholdAreaPercolation(\beta, r)$ from $[0,M]$ to $\R$  
is non-decreasing and Lipschitz. To this end we prove that there exists a constant $0<c\le 1$ such that for any $(\Intensity,\beta)$ be in $[0,M]^2$ and $u$ in $[0,1]$ 

\begin{equation}\label{stochregular} 
  P_{free}^{\Intensity,\beta+u} \preceq P_{wired}^{\Intensity,\beta } \quad \text{ and } \quad   P_{free}^{\Intensity,\beta} \preceq P_{wired}^{\Intensity +u,\beta +cu}.
\end{equation}

The expected regularity and monotony for  $\beta\to \IntensityThresholdAreaPercolation(\beta,r)$ follows from the threshold property given in Proposition \ref{propo_threshold_area}. Such stochastic domination  results  \eqref{stochregular} are consequences of uniform properties of the Papangelou intensity of the area-interaction measure defined for every $x\in \R^d$ and $\omega\in \ConfSpace$  by 
\begin{equation}\label{papangelouAREA}
\gamma_{\Intensity,\beta}(x,\Conf)
=
\Intensity e^{-\beta H_{\{x\}}(x \cup \Conf )}
=
\Intensity e^{-\beta \Leb (B_1(x)\setminus B_1(\Conf))}.
\end{equation}
Roughly speaking, this quantity is the quotient of the densities of the process with and without the point $x$. 
We refer to  \cite{georgii_kuneth} for rigorous definitions, interpretations and results on the topic.  By Theorem 1.1 in \cite{georgii_kuneth}, to prove \eqref{stochregular} it is sufficient to show that for any   $(\Intensity,\beta)$ be in $[0,M]^2$, $u$ in $[0,1]$, $x\in \R^d$ and  $\Conf \subset \Conf'$ in $ \ConfSpace$  

$$ \gamma_{\Intensity,\beta}(x,\Conf') \ge  \gamma_{\Intensity,\beta+u}(x,\Conf) \quad \text{ and } \quad \gamma_{\Intensity+u,\beta+cu}(x,\Conf') \ge  \gamma_{\Intensity,\beta}(x,\Conf).$$

The first inequality is obvious from the expression of  $\gamma_{z,\beta}$ in \eqref{papangelouAREA}. For the second inequality we define the constant 

$$c=\min \left(1,\frac{e^{-(M+1) v_d}}{ v_d M}\right)>0,$$ 

where $v_d$ is the volume of the unit ball in $\R^d$. Therefore

\begin{eqnarray*}
\gamma_{\Intensity+u,\beta+cu}(x,\Conf')&=&(\Intensity+u) e^{-(\beta+cu) H_{\{x\}}(\Conf' \cup x)} \\
&\ge & \gamma_{\Intensity,\beta}(x,\Conf')e^{-cu v_d}+ue^{-(M+1)v_d}\\
& \ge &  \gamma_{\Intensity,\beta}(x,\Conf)+ M(e^{-cu v_d}-1)+ue^{-(M+1)v_d}\\
& \ge & \gamma_{\Intensity,\beta}(x,\Conf)+ u\left(e^{-(M+1)v_d}-cM v_d\right)\\
& \ge & \gamma_{\Intensity,\beta}(x,\Conf).
\end{eqnarray*}

\section{Proof of Theorem \ref{theo_unicite_avant_threshold}}
\label{Section_proof_th_unicite}
In this section we are considering 
$\Intensity< \IntensityThresholdAreaPercolation(\beta,1)$, and we want to prove 
$P^{\Intensity, \beta}_{free}
=
P^{\Intensity, \beta}_{wired}$.
The proof relies on a disagreement coupling method.
\begin{definition}
A disagreement coupling 
$ \DisagreementCoupling_{\Lambda,\Conf^1_{\Lambda^c},
\Conf^2_{\Lambda^c}}$ index by a bounded $\Lambda \subseteq \R^\Dim$ and two configurations satisfying
$\Conf^1_{\Lambda^c} \subseteq\Conf^2_{\Lambda^c}$, is a coupling of two marginals, with canonical variables $\xi^1$ and $\xi^2$  satisfying
\begin{subequations}
\label{eq_dcf_def}
\begin{align}
\label{eq_dcf_marginal_gibbs}
\forall 1\le{}i\le{}2:\quad
&\DisagreementCoupling_{\Lambda,\Conf^1_{\Lambda^c},
\Conf^2_{\Lambda^c}}
(\xi^i=d \Conf')
=
\Specification^{\Intensity,\beta}_{\Lambda,\Conf^i_{\Lambda^c}}(d\Conf')
\\ 
\label{eq_dcf_domination}
&\DisagreementCoupling_{\Lambda,\Conf^1_{\Lambda^c},
\Conf^2_{\Lambda^c}}
(\xi^1 \subseteq \xi^2 )
=
1
\\ 
\label{eq_dcf_disagreement_cluster}
&\DisagreementCoupling_{\Lambda,\Conf^1_{\Lambda^c},
\Conf^2_{\Lambda^c}}
\left(
\forall x \in\xi^2 \setminus  \xi^1 | B_1(x) \Connected{B_1(\xi^2)} B_1(\Conf^2_{\Lambda^c}) 
\right)
=
1
\end{align}
\end{subequations}
\end{definition}
\begin{remark}
In the definition of a disagreement coupling from \cite{Hofer-temmel_Houdebert_2018}, there is a third marginal dominating the two first which is a Poisson point process.
But from the monotonicity property of area-interaction measures, see Proposition \ref{propo_dom_sto_poisson_area}, one can only consider two marginals in the coupling.
\end{remark}
\begin{proposition}
{Color{red} 
Assume $\Intensity< \IntensityThresholdAreaPercolation(\beta,1)$.
}
If there exists a disagreement coupling for every $\Lambda$ bounded and every configurations $\Conf^1_{\Lambda^c} \subseteq \Conf^2_{\Lambda^c}$, then $P^{\Intensity, \beta}_{free}
=
P^{\Intensity, \beta}_{wired}$.
\end{proposition}
\begin{proof}
Let $E$ be an event, that  without loss of generality, only depends on the configurations inside a given bounded $\Lambda$.
Then for $\Lambda \subseteq \Lambda_n$ we have
\begin{align*}
|P^{\Intensity, \beta}_{free}&(E) - P^{\Intensity, \beta}_{wired}(E) |
\\ &\leq  
\int \int 
|\Specification^{\Intensity,\beta}_{\Lambda_n,\Conf^1_{\Lambda_n^c}}(E)-
\Specification^{\Intensity,\beta}_{\Lambda_n,\Conf^2_{\Lambda_n^c}}(E)  |
P^{\Intensity, \beta}_{free}(d \Conf^1)
P^{\Intensity, \beta}_{wired}(d \Conf^2)
\\& = \int \int 
|\Specification^{\Intensity,\beta}_{\Lambda_n,\Conf^1_{\Lambda_n^c}}(E)-
\Specification^{\Intensity,\beta}_{\Lambda_n,\Conf^2_{\Lambda_n^c}}(E)  |
P^{\Intensity, \beta}_{free|\Lambda_{n+2}}(d \Conf^1)
P^{\Intensity, \beta}_{wired|\Lambda_{n+2}}(d \Conf^2),
\end{align*}

where $P^{\Intensity, \beta}_{free|\Lambda_{n+2}}$ (respectively $P^{\Intensity, \beta}_{wired|\Lambda_{n+2}}$) is the restriction of $P^{\Intensity, \beta}_{free}$ (respectively $P^{\Intensity, \beta}_{wired}$) on $\Lambda_{n+2}$. Now by the stochastic domination
$P^{\Intensity, \beta}_{free}
\preceq
P^{\Intensity, \beta}_{wired}$, we have from Strassen's theorem, see for instance 
\cite{lindvall_1999_strassen},  the existence of a thinning probability $\phi^{1}_{\Conf^2}$ such that
\begin{align*}
|&P^{\Intensity, \beta}_{free}(E) - P^{\Intensity, \beta}_{wired}(E) |
\\ &\leq  
\int \sum_{\scriptscriptstyle \Conf^1 \subseteq\Conf^2 } 
|\Specification^{\Intensity,\beta}_{\Lambda_n,\Conf^1_{\Lambda_n^c}}(E)-
\Specification^{\Intensity,\beta}_{\Lambda_n,\Conf^2_{\Lambda_n^c}}(E)  |
\phi^1_{\Conf^2}(\Conf^1)
P^{\Intensity, \beta}_{wired|\Lambda_{n+2}}(d \Conf^2)
\\ & \leq  
\int 
\sum_{\scriptscriptstyle \Conf^1 \subseteq\Conf^2 } 
\DisagreementCoupling_{\Lambda_n,\Conf^1_{\Lambda_n^c},
\Conf^2_{\Lambda_n^c},}
\left(
\Lambda \Connected{B_1(\xi^2)}\Lambda_{n-2}
\right)
\phi^1_{\Conf^2}(\Conf^1)  
P^{\Intensity, \beta}_{wired|\Lambda_{n+2}}(d \Conf^2),
\end{align*}
where the last inequality comes from the existence of the disagreement coupling  and the property \eqref{eq_dcf_disagreement_cluster}.
Therefore
\begin{align*}
| P^{\Intensity, \beta}_{free}(E) - P^{\Intensity, \beta}_{wired}(E) |
 & \leq  
\int 
\Specification^{\Intensity,\beta}_{\Lambda_n,\Conf_{\Lambda_n^c}}
\left(
\Lambda \Connected{B_1(\xi)}\Lambda_{n-2}
\right)
P^{\Intensity, \beta}_{wired|\Lambda_{n+2}}(d \Conf)
\\ & = P^{\Intensity, \beta}_{wired}
\left(
\Lambda \Connected{B_1(\xi)}\Lambda_{n-2}
\right)
\underset{n \to \infty}{\longrightarrow} 0,
\end{align*}
where the convergence is a consequence of
$ \Intensity < \IntensityThresholdAreaPercolation(\beta,1)$.
\end{proof}
It remains to prove the existence of the disagreement coupling.

\begin{proposition}
For each bounded $\Lambda$ and each 
$\Conf^1_{\Lambda^c} \subseteq \Conf^2_{\Lambda^c}$,
there exists a disagreement coupling $\DisagreementCoupling_{\Lambda,\Conf^1_{\Lambda^c},\Conf^2_{\Lambda^c}}$.
\end{proposition}
The construction of the disagreement coupling is a generalisation of the one made in \cite{Hofer-temmel_Houdebert_2018}, where the dominating measure is a Poisson point process.
The coupling is sampled starting from the balls close to the boundary of $\Lambda$, and going inductively inside $\Lambda$.
\begin{proof}
The coupling is constructed inductively.
Recall that $\Lambda \subseteq \R^\Dim$ is bounded and that 
$\Conf^1_{\Lambda^c} \subseteq \Conf^2_{\Lambda^c}$.
Define the \emph{disagreement zone}
\begin{align*}
\Gamma = \{
x \in \Lambda, ||x; \Conf^2_{\Lambda^c}||  \leq 2\}
\end{align*}
as the region where a point $x$ of the point process would be directly $1-connected$ to the  boundary condition $\Conf^2_{\Lambda^c}$ (i.e. the ball $B_1(x)$ would overlap $B_1(y)$ for at least one $y\in\Conf^2_{\Lambda^c}$).

The induction will be made with respect to the disagreement zone $\Gamma$ in the following way.
\begin{itemize}
\item 
If $\Gamma \not = \emptyset$, let us first sample 
$\xi^2 \sim \Specification^{\Intensity,\beta}_{\Lambda,\Conf^2_{\Lambda^c}}$.
We are then sampling $ \xi^1 \sim
\Specification^{\Intensity,\beta}_{\Lambda,\Conf^1_{\Lambda^c}}$ as a thinning of 
$\xi^2$.

This procedure is possible, since the condition
$\Conf^1_{\Lambda^c} \subseteq \Conf^2_{\Lambda^c}$ implies, thanks to Proposition \ref{propo_dom_sto_poisson_area}, the following domination:
$$
\Specification^{\Intensity,\beta}_{\Lambda,\Conf^1_{\Lambda^c}}
\preceq
\Specification^{\Intensity,\beta}_{\Lambda,\Conf^2_{\Lambda^c}}.
$$
From $\xi^1$ and $\xi^2$ we are only keeping the points inside $\Gamma$.
The induction then goes on with $\Lambda \leftarrow \Lambda \cap \Gamma^c$ with the new boundary conditions
$\Conf^1_{\Lambda^c \cup \Gamma} = \Conf^1_{\Lambda^c} \cup \xi^1_{\Gamma}$ and
$\Conf^2_{\Lambda^c \cup \Gamma} = \Conf^2_{\Lambda^c} \cup \xi^2_{\Gamma}$.
\item
If $\Gamma= \emptyset$.
This is the terminal step of the induction.
In this case we have 
$\Specification^{\Intensity,\beta}_{\Lambda,\Conf^1_{\Lambda^c}}
=
\Specification^{\Intensity,\beta}_{\Lambda,\Conf^2_{\Lambda^c}}
=
\Specification^{\Intensity,\beta}_{\Lambda,\emptyset}$.
Therefore we simply sample 
$\xi^1= \xi^2 \sim \Specification^{\Intensity,\beta}_{\Lambda,\emptyset}$.
\end{itemize}

It is easy to see that the induction terminates almost surely.
Indeed if at one step the sampled configuration $\xi^2$ is empty (which happens with positive bounded from below probability) then at the following step we will have $\Gamma= \emptyset$.
Therefore the number of steps is dominated by a geometric random variable, which is almost surely finite.

Finally the construction ensures that all properties of \eqref{eq_dcf_def} are fullfilled.
\end{proof}

\section{Annex} \label{section_annex}
\subsection{Proof of Proposition \ref{propo_revealment_borne} }
Let us define the algorithm $T_s$ which explores the $r$-connected components of $\partial \Lambda_s$.
\begin{definition}[Definition of the algorithm]
During the first $i$ steps, the configuration inside the cubes $\egras_1, \dots, \egras_i$ have been explored, and we write $Z_i= B_r(\Conf_{\egras_{[i]}})$ for the region known to be covered by $r$-balls. 
Remark that $Z_i$ might not be entirely connected (in $Z_i$) to $\partial \Lambda_s$, and wrote $Z_i^s$ for the subregion of $Z_i$ of points connected (in $Z_i$) to $\partial \Lambda_s$.


 At the step $i+1$ we take a cube $\egras_{i+1} \in\varepsilon (\Z+1/2) ^\Dim \cap \Lambda_{n+r+1}\setminus \egras_{[i]} $ such that
 $$ || \Delta_{\egras_{i+1}}^\varepsilon ; Z_i^s\cup \partial \Lambda_s || \leq r. $$
If no such $\egras_{i+1}$ exists, then the algorithm stops as the connected components of 
$\partial \Lambda_s$ have been entirely explored.
If there is several $\egras_{i+1}$ satisfying the condition, we choose one using a deterministic (but not important) rule.

We then look at the configuration inside the chosen cube $\egras_{i+1}$ and set 
$Z_{i+1}= Z_i \cup B_r(\Conf_{\egras_{i+1}})=
B_r(\Conf_{\egras_{[i+1]}})$.
\end{definition}
The bound \eqref{eq_revealment_borne} from Proposition \ref{propo_revealment_borne} follows directly from the definition of the algorithms.
\subsection{Proof of Lemma \ref{lemme_derivation_gibbs} }
We have 
$ P^{\Intensity, \beta}_{n,wired} (f) 
= \int
f(\Conf) \frac{z^{\#\Conf } 
h(\Conf)}
{\PartitionFunction(z, \beta,n,wired)} 
\Poisson_{\Lambda_n}(d\Conf)
$, where 
$$ h(\Conf):=
e^{-\beta \Leb(B(\Conf_{\Lambda_n} ,1) \cap \Lambda_{n-1} ) }.$$
Using a standard derivative theorem we obtain
\begin{align}
\label{eq_preuve_deriv_covariance_1}
\frac{d}{d\Intensity} P^{\Intensity, \beta}_{n,wired} (f)
=
\frac{1}{\Intensity} P^{\Intensity, \beta}_{n,wired} (f \times \#)
-
\frac{\frac{d}{d \Intensity}\PartitionFunction(z, \beta,n,wired)}
{\PartitionFunction(z, \beta,n,wired)}
P^{\Intensity, \beta}_{n,wired} (f).
\end{align}
Taking $f=1$ in \eqref{eq_preuve_deriv_covariance_1} yields
\begin{align*}
0 =
\frac{1}{\Intensity}
P^{\Intensity, \beta}_{n,wired} ( \#  )
-
\frac{\frac{d}{d \Intensity}\PartitionFunction(z, \beta,n,wired)}
{\PartitionFunction(z, \beta,n,wired)},
\end{align*}
which transforms \eqref{eq_preuve_deriv_covariance_1} into
\begin{align*}
\frac{d}{d\Intensity} P^{\Intensity, \beta}_{n,wired} (f)
=
\frac{1}{\Intensity} \left(
P^{\Intensity, \beta}_{n,wired} ( f \times \#  )
-
P^{\Intensity, \beta}_{n,wired} (  \#  )
P^{\Intensity, \beta}_{n,wired} (f)
\right),
\end{align*}
proving the result.

\vspace{1cm}
{\it Acknowledgement:} This work was supported in part by the Labex CEMPI  (ANR-11-LABX-0007-01), the GDR 3477 Geosto, the ANR project PPPP (ANR-16-CE40-0016), the ANR project MALIN (ANR-16-CE93-0003) and the 
Deutsche Forschungsgemeinschaft (DFG) - SFB1294/1 - 318763901.

\bibliographystyle{plain}
\bibliography{biblio}
\end{document}